\documentclass[12pt,a4paper]{article}     %Lis\"a\"a/poista draft tarpeen mukaan

\usepackage[utf8]{inputenc}
\usepackage[T1]{fontenc}
\usepackage{amsmath,amsfonts,amssymb,amsthm,amscd}
\usepackage{cite}
\usepackage{color}
\usepackage{fullpage}
\usepackage{graphicx}
\usepackage{hyperref}
\usepackage{mathtools}
\usepackage{verbatim}
\usepackage{cleveref}
\usepackage{authblk}

\theoremstyle{plain}
\newtheorem{theorem}{Theorem}[section]

\newtheorem{lemma}[theorem]{Lemma}
\newtheorem{corollary}[theorem]{Corollary}

\theoremstyle{definition}

\newtheorem{example}[theorem]{Example}

\theoremstyle{remark}
\newtheorem{remark}[theorem]{Remark}

\newcommand{\Q}{\mathbb Q}

\newcommand{\Z}{\mathbb Z}

\DeclareMathOperator{\lcm}{lcm}
\DeclareMathOperator{\ord}{ord}
\newcommand*\samethanks[1][\value{footnote}]{\footnotemark[#1]}

\title{A $p$-adic lower bound for a linear form in logarithms}

\author[1]{Neea Paloj\"arvi\thanks{The work of both authors was supported by The Emil Aaltonen Foundation.}}
\author[2]{Louna Sepp\"al\"a\samethanks}

\affil[1,2]{Department of Mathematics and Statistics, University of Helsinki, P.O.\ Box 68, 00014 Helsinki}
\affil[1]{neea.palojarvi@helsinki.fi}
\affil[2]{louna.seppala@helsinki.fi}

\date{}

\begin{document}

\maketitle

\vspace{-1cm}

%\listoftodos
%ABSTRACT
\begin{abstract}
Linear forms in logarithms have an important role in the theory of Diophantine equations. In this article, we prove explicit $p$-adic lower bounds for linear forms in $p$-adic logarithms of rational numbers using Padé approximations of the second kind.
\end{abstract}

%KEYWORDS

%MSC2020

%INTRODUCTION
\section{Introduction}

% \textbf{Vastaavuus Rhinin \& Toffinin ja meid\"an merkint\"ojen v\"alill\"a:}

% \vspace{0.5cm}

% \begin{tabular}{l|l}
% Rhin \& Toffin & Me \\
% \hline
% $m$ & $\mu$ \\
% $n$ & $k$ \\
% $k$ & $j$ \\
% $r$ & $m$ \\
% \end{tabular}

% \vspace{0.5cm}

When certain numbers are linearly independent, it is of interest to study how small their linear combination can be. Lower bounds for linear forms are closely tied to measures of irrationality and transcendence which are an essential part of transcendental number theory. Baker's famous work \cite{Baker1966, Baker1967a, Baker1967b} in generalising the Gel'fond-Schneider theorem started an entire branch of number theory concentrating on linear forms in logarithms. An important application of these is found in the theory of Diophantine equations, where linear forms in logarithms can be used to find bounds for the possible solutions (see, e.g., \cite{Bugeaud2002,Bugeaud2018,Bujacic2016,Gelfond1940}.)

%Traditionally interesting functions have been the exponential (see, e.g., \cite{Baker1965}, \cite{E-H2015}) and the logarithm (see, e.g, the introductory text \cite{Bugeaud2018}). Euler's factorial series has also been studied in the $p$-adic case (see, e.g., \cite{Bertrand2004}, \cite{Seppala2020}). 

The $p$-adic theory of logarithmic forms has also been developed regularly. In 1935, Mahler \cite{Mahler1935} established the $p$-adic analogue of the Gel'fond-Schneider theorem. Gel'fond \cite{Gelfond1940} proved a quantitative estimate for linear forms in two $p$-adic logarithms, later refined by Schinzel \cite{Schinzel1967}. Several researchers have made advances in estimates for linear forms in an arbitrary number of $p$-adic logarithms; see, e.g., \cite{Ballay2019,Bugeaud2002,Dong1995,Kaufman1971,Yu1989,Yu1990,Yu1994,Yu1998,Yu1999,Yu2007}. Some of the results are even explicit: In a series of papers, Yu \cite{Yu1989,Yu1990,Yu1994,Yu1998,Yu1999,Yu2007} proved lower bounds for $p$-adic linear forms in logarithms in an algebraic number field. Further, V\"a\"an\"anen and Xu \cite{Vaananen1988} proved lower bounds for linear forms in $G$-functions, including logarithms, in certain rational points in the $p$-adic case. 

In this article, we derive an explicit lower bound for the $p$-adic absolute value of a linear form in $p$-adic logarithms of given rational numbers by using Pad\'e approximations. Our results are of type
\begin{equation}
\label{eq:ourForm}
\left|\Lambda_p\right|_p >c H^{-\omega},
\end{equation}
where $\Lambda_p$ is a linear form in $p$-adic logarithms, $H$ is the largest absolute value of the coefficients of the linear form, $c$ is a positive constant which does not depend on $H$, and $\omega$ is an exponent which may depend on $H$. In some sense we present a $p$-adic analogue to the work of Rhin and Toffin \cite{Rhin1986}, where they study logarithms of algebraic numbers in the Archimedean case. Our work can also be seen as a continuation of the work by Heimonen, Matala-aho and V\"a\"an\"anen \cite{Heimonen1993} where they proved effective results for linear forms containing one logarithm.  

As, for example, Flicker mentions \cite[p.\ 397]{Flicker1977}, the best possible exponent for $H$ is of the form $-m-1-\varepsilon$, where $m+1$ is the number of terms in the linear form $\Lambda_p$ and $\varepsilon$ is an arbitrarily small real number when $H$ is large enough. We obtain this result in some of the cases (see \Cref{corollary:m1first} and \Cref{thm:mainresult4}) even though the exponents may not always be the best possible.

Compared to V\"a\"an\"anen and Xu \cite[Corollary 2]{Vaananen1988}, our results do not depend on that many different terms. Also, in order to apply our results, there is no need to know the functions' representations as $G$-functions or their system of linear differential equations as they are needed in order to apply V\"a\"an\"anen and Xu's result. The detailed formulation of V\"a\"an\"anen and Xu's Corollary 2 and some more precise comparison can be found in Appendix \ref{app:VaananenXu}.
%You can see a more detailed version of the result formulated in ???? containing also a little bit more precis comparison.

In his work (see e.g. \cite{Yu2007}), Yu considered more general cases and hence the bounds proved by Yu are considerably smaller in the case of rational numbers compared to our results. See Appendix \ref{app:Yu} for more details.

We approach the problem as follows: In \Cref{sec:results}, the results are described in a more detailed way. \Cref{sec:generalidea} contains the general idea of the proofs. \Cref{sec:Pade,sec:det,sec:estimates,sec:Contradiction} are devoted to some preliminary results which are used in \Cref{sec:proofs} to prove the main theorems. At the end of the paper, in \Cref{sec:examples}, we give some examples of the results.

%RESULTS
\section{Results}
\label{sec:results}

In this section we formulate the main results. First of all a few definitions are needed:

Fix a prime number $p$. The logarithm series
$$
\boldsymbol{\log}(1+t) = \sum_{n=0}^\infty \frac{(-1)^n}{n+1} t^{n+1}
$$
converges $p$-adically when $|t|_p < 1$. The boldface $\boldsymbol{\log}(x)$ is used to denote the $p$-adic logarithm, whereas non-bolded logarithms are usual real valued logarithms with base $e$.

Let $\alpha_1, \ldots, \alpha_m$ be rational numbers and $\lambda_0, \lambda_1, \ldots, \lambda_m$ be integers. We denote
$$
\Lambda_p := \lambda_0 + \lambda_1 \boldsymbol{\log} (1+\alpha_1) + \ldots + \lambda_m  \boldsymbol{\log}  (1+\alpha_m)
$$
and
%Definitions of M, alpha, H
\begin{equation*}
M \coloneqq \max_{1 \le i \le m} \left|\alpha_i \right|, \quad \alpha \coloneqq \max_{1 \le i \le m} \left|\alpha_i\right|_p, \quad\text{and}\quad H \coloneqq \max_{0 \le i \le m} \left|\lambda_i \right|.
\end{equation*}
Throughout the article, we assume $\alpha <1$ and thus $M \neq 1$.
Furthermore, let $Q>0$ be an integer such that $Q\alpha_i \in \Z$ for all $i=1,2,\ldots, m$.
%Definition for f
Denote also
\begin{equation}
\label{eq:fMQa}
    f(m, M, Q, \alpha)\coloneqq 2^{-1}\cdot 3^{-m}e^{-1.03883m}  Q^{-m} M^{-m} \alpha^{-m-1}.
\end{equation}

%Results in case M>1
\subsection{Case $M>1$}
\label{sec:resultsMgreater}

%\begin{equation}
%\label{eq:cMlarge}
%\begin{split}
%    c_1(m, M, Q, \alpha)
%    \coloneqq
    %\begin{cases}
     %   \begin{split}
 %    &\left(2^{m-1}(m+1)e^{1.03883m}Q^m \frac{M^{m+1}}{M-1}\right)^{-1} \\
 %    &\; \cdot \left( f(m, M, Q, \alpha)\alpha^{m+1}\right)^{1-\frac{2\log \log \left( f(m,M,Q,\alpha) \right)}{\log \left( f(m,M,Q,\alpha) \right)}}
%\end{split} & \text{if \eqref{eq:condH1}, \eqref{eq:condW1} hold}; \\
%\begin{split}
%   % &\left(2^{m+1}3^{2m}(m+1)e^{3.11649m} \vphantom{\frac{M^{3m+1}}{M-1}} \right. \\
  % & \left. \; \cdot Q^m(m+2) \frac{M^{3m+1}}{M-1}\right)^{-1}
%  \left(6^{m}(m+1)e^{2.07766m}Q^{2m} \frac{M^{2m+1}}{M-1}\right)^{-1}
%\end{split}
%&\text{if \eqref{eq:condW1} does not hold}
        % & \begin{split}
         %   &\text{if \eqref{eq:condW1} does not hold or} \\
          %  &\text{\eqref{eq:condH1}, \eqref{eq:condW1} hold but \eqref{eq:cond2} not}
        %\end{split}
  %  \end{cases}
%\end{split}    
%\end{equation}
%and 
%\begin{equation}
%\label{eq:omegaMlarge}
%\begin{split}
%   \omega_1(m, M, Q, \alpha) &\coloneqq 
   % \begin{cases}
%    1-\frac{14 \log \left( f(m, M, Q, \alpha)\alpha^{m+1} \right)}{\log \left( f(m,M,Q,\alpha) \right)}.% & \text{if \eqref{eq:condH1}, \eqref{eq:condW1} hold}; \\
   % 1 & \text{if \eqref{eq:condW1} does not hold},
   % \begin{split}
           % &\text{if \eqref{eq:condW1} does not hold or} \\
           % &\text{\eqref{eq:condH1}, \eqref{eq:condW1} hold but \eqref{eq:cond2} not}
       % \end{split}
   % \end{cases}
%\end{split}
%\end{equation}

In this section, we describe the results for the case $M>1$. Denote
\begin{multline}\label{def:R}
    R_1(m,M,Q,\alpha) = 1.03883m +m\log Q+ \log(m+1) + (m-1)\log 2 \\
    + (m+1)\log M - \log(M-1) + \log(m+2)+ \log \alpha -\log \log f(m, M, Q, \alpha),
\end{multline}
\begin{equation*}\label{eq:cMlarge}
\begin{aligned}
    c_1(m,M, Q, \alpha)= \frac{M-1}{6^{m}(m+1)e^{2.07766m} Q^{2m}M^{2m+1}} \cdot e^{\left(\frac{(m+1)\log \alpha}{\log f(m,M,Q,\alpha)}+1\right)\left(R_1(m,M,Q,\alpha)+1\right)},
\end{aligned}    
\end{equation*}
and
\begin{equation*}\label{eq:omegaMlarge}
\begin{aligned}
     \omega_1(m, M, Q, \alpha, H)= \;&\left(-\frac{(m+1) \log\alpha}{\log f(m,M,Q,\alpha)}-1\right) \\
     &\cdot\left(1+\frac{\log\left(\log H+R_1(m,M,Q,\alpha)\right)}{\log H}\right)+1.
\end{aligned}     
\end{equation*}

Using this notation, the first main result is stated as follows:
%Mainresult2, M>1
\begin{theorem}
\label{thm:mainresult2}
Let $m \ge 1$. Let $\lambda_0, \lambda_1, \ldots, \lambda_m \in \Z$ and $\alpha_1, \ldots, \alpha_m \in \Q$ be given numbers such that $\lambda_i \neq 0$ for some $i$ and the numbers $\alpha_i$ are pairwise distinct and non-zero. Assume also that $M >1$, $\alpha <1$, $H>1$,
\begin{equation}
\label{eq:assumpless1}
f(m,M,Q,\alpha)  >1
\end{equation}
and
\begin{equation}
\label{eq:condW1}
-\frac{1}{e^{R_1(m,M,Q,\alpha)}H} = \frac{-\log f\left(m,M,Q,\alpha  \right)}{2^{m-1} e^{1.03883m} Q^m H \cdot \frac{M^{m+1}}{M-1} \cdot (m+1)(m+2) \alpha} \ge -\frac{1}{e}.
\end{equation}

Then
$$
| \Lambda_p |_p > c_1(m,M,Q,\alpha)H^{-\omega_1(m, M,Q, \alpha, H)}.
$$
\end{theorem}

We would like to give a more concrete idea of how the exponent in Theorem \ref{thm:mainresult2} looks like when the number $H$ is large enough. Hence, in the next corollaries we provide some results which follow from Theorem \ref{thm:mainresult2} and may describe the result better but are not as sharp as Theorem \ref{thm:mainresult2}. The first of these corollaries emphasizes that the exponent of $H$ in Theorem \ref{thm:mainresult2} essentially consists of a term which does not depend on $H$ plus a term of size $\log\log H/\log H$ when $H \to \infty$:

\begin{corollary}
\label{corollary:loglogH}
Assume that the conditions of Theorem \ref{thm:mainresult2} hold, and also that 
$$
H\ge \max\left\{3, e^{R_1(m,M,Q,\alpha)}\right\}. 
$$
Then
$$
\left|\Lambda_p\right|_p>c_1(m,M,Q,\alpha)H^{\frac{(m+1)\log\alpha}{f(m,M,Q,\alpha)}+11.633\cdot\left(1+\frac{(m+1)\log\alpha}{f(m,M,Q,\alpha)}\right)\frac{\log\log H}{\log H}}.
$$
\end{corollary}

For any real number $\varepsilon$, we have $|\Lambda_p|_p \gg H^{-1-m-\varepsilon}$, and, as mentioned in the Introduction, this is the best possible exponent for $H$. The next corollary shows that our \Cref{thm:mainresult2} leads to this estimate in some cases. 
%In the result we denote
%\begin{equation*}
%    g\coloneqq g(\varepsilon,m,M,Q,\alpha) \coloneqq \frac{3}{\varepsilon}\left(\frac{(m+1)\log\alpha}{\log f(m,M,Q,\alpha)}+1\right).
%\end{equation*}

%m+1+epsilon for M>1
\begin{corollary}
\label{corollary:m1first}
Let $\varepsilon \in (0,3]$ be a real number. Assume that the numbers $m$, $\lambda_i$ ($i=0,1,\ldots, m$), $\alpha_j$ ($j=0,1,\ldots, m$), $M$, $Q$, $\alpha$ and $H$ satisfy all of the conditions in \Cref{thm:mainresult2}. Suppose also that
\begin{equation}
\label{eq:m1epsilonpart}
    \frac{(m+1)\log\alpha}{\log f(m,M,Q,\alpha)} \ge -m-1-\frac{\varepsilon}{3}
\end{equation}
and
\begin{equation}
\label{eq:condHm1Big}
    H \ge \max\left\{\left(c_1(m,M,Q,\alpha)\right)^{-3/\varepsilon}, \frac{1}{2}e^{-\left(3m/\varepsilon+1\right)W_{-1}\left(-2^{-\varepsilon/(3m+\varepsilon)}\frac{\varepsilon}{3m+\varepsilon}\right)}\right\}.
\end{equation}

Then
$$
\left|\Lambda_p\right|_p > H^{-m-1-\varepsilon}.
$$
\end{corollary}

\begin{remark}
The assumption $\varepsilon \le 3$ is not necessary in Corollary \ref{corollary:m1first}; it is included only to relax the assumptions for the number $H$.
\end{remark}

%First remark of existence, M>1
\begin{remark}
\label{rmk:analysis1}
We note that the conditions of \Cref{corollary:m1first} can be satisfied:

First of all, if $\alpha$ is small enough, the terms $\varepsilon$ and $m$ do not change and the term $Q$ is bounded. Then the expression $1/f(m,M,Q,\alpha)$ is essentially of the form $c \alpha$, where $c$ is a constant. Since we also have
\begin{equation*}
    \lim_{x \to 0} \frac{(m+1)\log(x)}{\log(c x)}=(m+1)
\end{equation*}
for all constants $c>0$, condition \eqref{eq:m1epsilonpart} is satisfied when $\alpha$ is small enough.

After finding an $\alpha$ small enough, we also know the terms $\varepsilon, m,M$ and $Q$. Hence we can find a number $H$ large enough, and all of the conditions in \Cref{corollary:m1first} are satisfied. 

This is illustrated in \Cref{example:pm1}.
\end{remark}

\Cref{thm:mainresult2} and \Cref{corollary:m1first} are quite long and may be a little bit difficult to interpret, so we also mention the following corollary giving a shorter but not that sharp result:

%Corollary for M>1, simplified version
\begin{corollary}
\label{corollary:main2Simp}
Let $m$, $\lambda_0, \lambda_1, \ldots, \lambda_m$, $\alpha_1, \ldots, \alpha_m$, $M$, and $\alpha$ be defined as in \Cref{thm:mainresult2}. Further, assume that condition \eqref{eq:condW1} holds,
\begin{equation}
\label{eq:condlogf}
\log f(m,M,Q, \alpha) \ge 1
\end{equation}
and
\begin{equation*}\label{eq:condH1}
    \log H \ge \max\left\{1.03883m, (m+1)\log M, \log(m+2), m\log Q\right\}.
\end{equation*}

Then 
$$
| \Lambda_p |_p > \frac{M-1}{6^{m}(m+1)e^{2.07766m}Q^{2m}M^{2m+1}}H^{-1-14\log 2-14m\log(3e^{1.03883}QM)}.
$$
\end{corollary}

%Second remark of existence, M>1
\begin{remark}
Using similar reasoning as in \Cref{rmk:analysis1}, it can be seen that the case in \Cref{corollary:main2Simp} is possible.
\end{remark}

%Results in case M<1
\subsection{Case $M<1$}
\label{sec:resultsSmall}

In this section, we describe our results for the case $M<1$. Since a similar type of result as in \Cref{thm:mainresult2}, where we do not have additional conditions for the number $H$ or for the term $f(m,M,Q,\alpha)$, would be quite long in this case, we do not present it. Instead, we formulate analogous results to \Cref{corollary:m1first,corollary:main2Simp}. A result similar to \Cref{thm:mainresult2} can be obtained from the proof of \Cref{thm:mainresult3}, and this is left to the interested reader.

For $M< 1$, we write
\begin{equation*}\label{eq:cMsmall}
c_2(m, M, Q, \alpha) :=
%\begin{cases}
%\left( Q^{2m} 2^{m+1} 3^m m(m+1)e^{2.07766m} \right)^{-1}, \text{when \eqref{eq:condH5} does not hold};\\
\left(26m(m+1)\left(6Q^2e^{2.07766}\right)^me^{\frac{5\log \left( 2 \cdot 3^m Q^m e^{1.03883m} \right)}{\log f(m, M, Q, \alpha)}}\right)^{-1}%\\
%\text{when \eqref{eq:condH5} and \eqref{eq:condH3} hold}.
%\end{cases}
\end{equation*}
and
\begin{equation*}\label{eq:omegaMsmall}
\omega_2(m, M, Q, \alpha, H) := 
%\begin{cases}
%1, &\text{when \eqref{eq:condH5} does not hold}; 
1+\frac{\log\log H}{\log H}+\frac{11\log \left( 2 \cdot 3^m Q^m e^{1.03883m} \right)}{\log f(m, M, Q, \alpha)}.
%&\text{when \eqref{eq:condH5} and \eqref{eq:condH3} hold}.
%\end{cases}
\end{equation*}

The following result is similar to \Cref{corollary:main2Simp}:

% M<1 result
\begin{theorem}\label{thm:mainresult3}
Let $m \ge 1$. Let $\lambda_0, \lambda_1, \ldots, \lambda_m \in \Z$ and $\alpha_1, \ldots, \alpha_m \in \Q$ be given numbers such that $\lambda_i \neq 0$ for some $i$ and the numbers $\alpha_i$ are pairwise distinct and non-zero. Assume also that $M < 1$, $\alpha <1$, condition \eqref{eq:condlogf} holds,
\begin{equation}
\label{eq:condH5}
\frac{\log^2 f\left(m,M,Q,\alpha \right)}{2^m e^{1.03883m}Q^m Hm(m+1)(m+2) \alpha} \le \frac{4}{e^2},
\end{equation}
and
\begin{equation}
\label{eq:condH3}
\log H > \max \left\{ \left( \frac{m}{2}+1 \right) \log 2, 0.519415m, \frac{m}{2} \log Q, \log (m+2), e \right\}.
\end{equation}

Then
$$
| \Lambda_p |_p > c_2(m,M,Q,\alpha)H^{-\omega_2(m, M,Q, \alpha, H)}.
$$
\end{theorem}

% Remark of m+1+epsilon, M<1
\begin{remark}
Using same kind of reasoning as in the proofs of \Cref{corollary:m1first} and \Cref{thm:mainresult4} and noting that the case is possible as in \Cref{rmk:analysis1}, we see that the previous result is of size $H^{-11m-1-\varepsilon}$ when $H$ is large and $\alpha$ small enough. The extra coefficient $11$ comes from simplification of the results. As we shall see in the next theorem, it can be removed.

\end{remark}

Next we formulate the result in the form where the exponent of $H$ is of the best possible form $m+1+\varepsilon$. For that, we denote
\begin{align*}
        R_2(m,Q,\alpha) \coloneqq \;&\left( \frac{m}{2}+1 \right) \log 2 + 0.519415m + \frac{m}{2} \log Q \\
&+ \frac{1}{2} \left( \log m + \log (m+1) + \log (m+2) + \log \alpha \right).
\end{align*}

%m+1+epsilon for M<1
\begin{theorem}
\label{thm:mainresult4}
Let $\varepsilon \in (0,3]$ be a real number. Assume that the numbers $m$, $\lambda_i$ ($i=0,1,\ldots, m$), $\alpha_j$ ($j=0,1,\ldots, m$), $M, Q, \alpha$, and $H$ satisfy the conditions given in \Cref{thm:mainresult3}, including condition \eqref{eq:condH5}. In addition, suppose that
\begin{equation}
\label{eq:flogmeps}
\frac{\log(M^m\alpha^{m+1})}{\log f(m, M, Q, \alpha)} \ge -m-1-\frac{\varepsilon}{3},
\end{equation}
\begin{equation}
\label{eq:assumpHepsilon}
\begin{aligned}
    H \ge \;& \left(2^{m+4+2R_2(m,Q,\alpha)}\cdot9^{R_2(m,Q,\alpha)+1}\left(e^{1.03883}Q\right)^{(2R_2(m,Q,\alpha)+3)m} m(m+1)\right)^{3/\varepsilon},
\end{aligned}
\end{equation}
and
\begin{equation}
\label{eq:assumpHepsilonSecond}
\begin{aligned}
    H \ge e^{-\frac{6(m+1+\varepsilon/3)}{\varepsilon}W_{-1}\left(-\frac{\varepsilon}{6(m+1+\varepsilon/3)}\right)}.
\end{aligned}
\end{equation}

Then
$$
|\Lambda_p|_p > H^{-m-1-\varepsilon}.
$$
\end{theorem}

\begin{remark}
Similarly as in the case of \Cref{corollary:m1first}, the assumption $\varepsilon \le 3$ is not necessary.
\end{remark}

%COMPARISONS
%\section{Comparison to the results of others}\label{sec:comparison}
%{\color{red} LUKU 3 TULEE TODENN\"aK\"oISESTI POISTUMAAN, VERTAILU SIS\"aLLYTET\"a\"aN JOHDANTOON.}

%GENERAL IDEA
\section{Overview of the proofs}
\label{sec:generalidea}

Before going to the actual proofs, let us look at the general idea behind them. The key concepts are Padé approximations for the individual functions and the $p$-adic product formula. This method is familiar from earlier works in this area (see, e.g., \cite{Heimonen1993,Vaananen1988}), but since the general idea is rarely explained, we took the opportunity to do so here.

Our target is to bound $\left|\Lambda_p\right|_p$ from below.
First, we derive Pad\'e approximations for the terms $\boldsymbol{\log}(1+\alpha_j)$; this is done in \Cref{sec:Pade}. For $k, \mu, j \in \Z$, $0 \le \mu \le m$, and $j \in [1,m]$, we get
\begin{equation*}
B_{k,\mu,0}(t) \cdot \boldsymbol{\log}(1+\alpha_jt) - B_{k,\mu,j}(t) = S_{k,\mu,j}(t),
\end{equation*}
where $B_{k,\mu,0}(t)$, $B_{k,\mu,j}(t)$, and $S_{k,\mu,j}(t)$ are certain series such that the numbers $Q^{mk+m}B_{k,\mu,0}(1)$ and $Q^{mk+m}B_{k,\mu,0}(1)$ are integers. The idea is to write the expression $\Lambda_p$ using the previous functions and then estimate these functions. The estimates for the terms $B_{k,\mu,0}(1)$, $B_{k,\mu,j}(1)$, and $S_{k,\mu,j}(1)$ are derived in \Cref{sec:estimates}. The general idea how these estimates are used is described below and done in detail in \Cref{sec:proofs}.

We can write 
\begin{equation*}\label{eq:Wdef}
\begin{aligned}
Q^{mk+m}B_{k,\mu,0}(1)\Lambda_p &=\sum\limits_{j=0}^m \lambda_jQ^{mk+m}B_{k,\mu,j}(1)-\sum\limits_{j=1}^m \lambda_jQ^{mk+m}S_{k,\mu,j}(1) \\
&\coloneqq T(k, \mu)-\sum\limits_{j=1}^m \lambda_jQ^{mk+m}S_{k,\mu,j}(1),
\end{aligned}
\end{equation*}
where we have multiplied by $Q^{mk+m}$ to obtain integers.
%Further, later in \Cref{sec:Pade} we see that
%by \cite[Proof of Lemma 5]{Rhin1986}
%\begin{equation*}
%\begin{aligned}
%P_{0,k,\mu}(1)&=\sum_{i_1=0}^k\ldots\sum_{i_m=0}^k \binom{\mu+k+i_1+\ldots+i_m}{k}\binom{k}{i_1}\cdots\binom{k}{i_m}\cdot \\
%&\quad\cdot \alpha_1^{k-i_1}\cdots \alpha_m^{k-i_m}  
%\end{aligned}
%\end{equation*}
%and
%\begin{equation*}
%\begin{aligned}
%P_{j,k,\mu}(1)&=\sum_{i_1=0}^k\ldots\sum_{i_m=0}^k \binom{\mu+k+i_1+\ldots+i_m}{k}\binom{k}{i_1}\cdots\binom{k}{i_m}\cdot \\
%&\quad\cdot \alpha_1^{k-i_1}\cdots \alpha_m^{k-i_m} \sum_{s=1}^{\mu+i_1+\ldots+i_m} \frac{(-1)^{s-1}}{s}\alpha_j^s.
%\end{aligned}
%\end{equation*}
According to the previous paragraph, $T(k,\mu) \in \Z$, and in \Cref{sec:det} we show that for every $k$ there exists a $\mu$ such that $T(k,\mu) \ne 0$. We would like to find a lower bound for the $p$-adic absolute value of the term $Q^{mk+m}B_{k,\mu,0}(1)\Lambda_p$, and continue by showing that $|T(k,\mu)|_p = |Q^{mk+m}B_{k,\mu,0}(1) \Lambda|_p$. Then it is sufficient to estimate the term $T(k,\mu)$.

%Choose the number $\mu$ so that $T(k,\mu) \ne 0$. (Such a $\mu$ indeed exists for any positive $k$; see \Cref{sec:det}.) Since the $p$-adic absolute value of a sum is at most as big as the maximum of the terms in the sum, we prove that the maximum is $|T(k,\mu)|_p$.

We make the assumption that
$$
|B_{k,\mu,0}(1) \Lambda_p|_p \le | \lambda_1 S_{k,\mu,1}(1) + \ldots + \lambda_m S_{k,\mu,m}(1) |_p.
$$ 
This implies $|T(k,\mu)|_p \le | \lambda_1 S_{k,\mu,1}(1) + \ldots + \lambda_m S_{k,\mu,m}(1) |_p$. Then, using the product formula and the estimates for the terms $B_{k,\mu,j}(1)$ and $S_{k,\mu,j}(1)$ proved in \Cref{sec:estimates}, we get
\begin{align*}
1 = \;&|T(k,\mu)| \prod_{q \in \mathbb{P}} |T(k,\mu)|_q \\
\le \;&|T(k,\mu)||T(k,\mu)|_p \\
\le \;&Q^{mk+m}|\lambda_0 B_{k,\mu,0}(1) + \ldots + \lambda_m B_{k,\mu,m}(1)|\cdot |\lambda_1 S_{k,\mu,1}(1) + \ldots + \lambda_m S_{k,\mu,m}(1)|_p \\
\le \;&Q^{mk+m}\left( \sum_{i=0}^m |\lambda_i| \right) \max_{0 \le j \le m} \{ |B_{k,\mu,j}(1)| \} \cdot  \max_{1 \le j \le m} \{|S_{k,\mu,j}(1)|_p\}. \\
\le \;&\ldots \\
< \;&1
\end{align*}
when $k$ is large enough, which is a contradiction. Therefore we must have 
$$
|B_{k,\mu,0}(1) \Lambda_p|_p > | \lambda_1 S_{k,\mu,1}(1) + \ldots + \lambda_m S_{k,\mu,m}(1) |_p
$$ 
for big $k$, from which it follows that $|T(k,\mu)|_p = |Q^{mk+m}B_{k,\mu,0}(1) \Lambda|_p$. So we get
$$
1 \le |T(k,\mu)||T(k,\mu)|_p = |T(k,\mu)| |Q^{mk+m}B_{k,\mu,0}(1) \Lambda_p|_p \le |T(k,\mu)||\Lambda_p|_p.
$$
This gives a lower bound for $|\Lambda_p|_p$ when we just bound the term $|T(k,\mu)|$.

A detailed proof for the previous contradiction is done in \Cref{sec:Contradiction}, and the proofs for the main results can be found in \Cref{sec:proofs}.

%PADÉ-APPROXIMATIONS
\section{Pad\'e approximations}
\label{sec:Pade}

First, we mention a useful lemma which is needed later.
Let $m, k \in \Z_{\ge 1}$. For $i=0,\ldots,mk$, we define
$$
\sigma_i = \sigma_i \left(k, \overline{\alpha} \right) := (-1)^i \sum_{i_1 + \ldots + i_m = i} \binom{k}{i_1} \cdots \binom{k}{i_m} \cdot \alpha_1^{k-i_1} \cdots \alpha_m^{k-i_m}.
$$

\begin{lemma}
\label{sigmalemma}
\cite[Lemma 4.1, estimate (15)]{Seppala2020}
For $\sigma_i$ defined as above, we have
$$
\sum_{i=0}^{mk} |\sigma_i|t^i \le \prod_{j=1}^m \left(|\alpha_j|+t\right)^k.
$$
\end{lemma}

Now we derive the Pad\'e approximations. The following proof is similar to that of \cite[Theorem 2.2]{Matala-aho2011}.

\begin{lemma}
Let $m, k \in \Z_{\ge 1}$ and $\mu \in \{0,1,\ldots,m\}$. Set
$$
A_{k,\mu,0}(z) := (-1)^{mk} \sum_{i=0}^{mk} \binom{i+k+\mu}{k} \sigma_i\left(k, \overline{\alpha}\right) z^{mk-i}.
$$
Then there exist polynomials $A_{k,\mu,j}(z)$ and remainders $R_{k,\mu,j}(z)$, where $j=1,\ldots,m$, defined as 
\begin{equation}
\label{def:Aj}
    A_{k,\mu,j}(z) \coloneqq \sum_{N=0}^{mk+\mu-1} \sum_{i=mk-\min\{N,mk\}}^{mk} (-1)^{mk} \binom{i+k+\mu}{k} \frac{\sigma_{i}\left(k, \overline{\alpha}\right) \alpha_j^{N-mk+i}}{N-mk+i+1} z^N
\end{equation}
and
\begin{equation}
\label{def:Rj}
    R_{k,\mu,j}(z) \coloneqq \sum_{N=mk+k+\mu}^\infty \sum_{i=0}^{mk} (-1)^{mk} \binom{i+k+\mu}{k} \frac{\sigma_{i}\left(k, \overline{\alpha}\right) \alpha_j^{N-mk+i}}{N-mk+i+1} z^N,
\end{equation}
for which we have
\begin{equation}\label{padeeq1}
A_{k,\mu,0}(z) \cdot \frac{\boldsymbol{\log}(1-\alpha_jz)}{-\alpha_jz} - A_{k,\mu,j}(z) = R_{k,\mu,j}(z)
\end{equation}
with
\begin{equation}\label{degreecond}
\begin{cases}
\deg A_{k,\mu,0}(z) = mk; \\
\deg A_{k,\mu,j}(z) \le mk + \mu -1; \\
\underset{t=0}{\ord} \, R_{k,\mu,j}(z) \ge mk+k+\mu.
\end{cases}
\end{equation}
\end{lemma}

\begin{proof}
Let us prove the claim by considering the first term on the left-hand side of \eqref{padeeq1}. We have
$$
\frac{\boldsymbol{\log}(1-\alpha_jz)}{-\alpha_jz} = \sum_{n=0}^\infty \frac{(\alpha_jz)^n}{n+1}.
$$
Let us write $A_{k,\mu,0}(z) = (-1)^{mk} \sum_{h=0}^{mk} \binom{mk-h+k+\mu}{k} \sigma_{mk-h}\left(k, \overline{\alpha}\right) z^h$. Then
$$
A_{k,\mu,0}(z) \cdot \frac{\boldsymbol{\log}(1-\alpha_jz)}{-\alpha_jz} = \sum_{N=0}^\infty r_{N,j} z^N,
$$
where
\begin{align*}
r_{N,j} &= \sum_{n+h=N} (-1)^{mk} \binom{mk-h+k+\mu}{k} \frac{\sigma_{mk-h}\left(k, \overline{\alpha}\right) \alpha_j^n}{n+1}\\
&= \sum_{h=0}^{\min \{N,mk\}} (-1)^{mk} \binom{mk-h+k+\mu}{k} \frac{\sigma_{mk-h}\left(k, \overline{\alpha}\right) \alpha_j^{N-h}}{N-h+1}\\
&= \sum_{i=mk-\min\{N,mk\}}^{mk} (-1)^{mk} \binom{i+k+\mu}{k} \frac{\sigma_{i}\left(k, \overline{\alpha}\right) \alpha_j^{N-mk+i}}{N-mk+i+1}. \\
\end{align*}
These give expressions \eqref{def:Aj} and \eqref{def:Rj}.

To fulfil \eqref{degreecond}, we need to show that $r_{N,j}=0$ when $mk+\mu \le N \le mk+\mu+k-1$. So let $N=mk+\mu+a$, where $0 \le a \le k-1$. Then the expression
\begin{align*}
\binom{i+k+\mu}{k} \cdot \frac{1}{N+i-mk+1} &= \frac{(i+\mu+k)!}{k!(i+\mu)!} \cdot \frac{1}{i+\mu+a+1} \\
&= \frac{(i+\mu+k)(i+\mu+k-1) \cdots (i+\mu+1)}{k!(i+\mu+a+1)} \\
&=: P(i)
\end{align*}
is a polynomial in $i$ of degree $k-1$. From the properties of the coefficients $\sigma_i$ it follows that $\sum_{i=0}^{mk} P(i) \cdot \alpha_j^i \cdot \sigma_{i}\left(k, \overline{\alpha}\right) = 0$ (see \cite[Lemma 2.1]{Matala-aho2011} or \cite[Lemma 4.1]{Seppala2020}). Now we have $r_{N,j} =0$ 
when $mk+\mu \le N \le mk+\mu+k-1$,
%Thus, in particular, $r_{N,j}=0$ when $mk+\mu \le N \le mk+\mu+k-1$.
and the proof is done.
\end{proof}

Let us finally multiply both sides of equation \eqref{padeeq1} by $d_{mk+\mu}$ to ensure that the numerator polynomials $A_{k,\mu,j}(z)$ have integer coefficients too. Denote
\begin{equation}
\label{def:B0}
B_{k,\mu,0}(t) := d_{mk+\mu} A_{k,\mu,0}(-t),
\end{equation}
\begin{equation}
\label{def:Bj}
B_{k,\mu,j}(t) := d_{mk+\mu} \alpha_jt A_{k,\mu,j}(-t),
\end{equation}
and
\begin{equation}\label{eq:remainder}
S_{k,\mu,j}(t) := d_{mk+\mu} \alpha_jt R_{k,\mu,j}(-t).
\end{equation}
Then our approximation functions become
\begin{equation*}\label{padeeq2}
B_{k,\mu,0}(t) \cdot \boldsymbol{\log}(1+\alpha_jt) - B_{k,\mu,j}(t) = S_{k,\mu,j}(t).
\end{equation*}
\section{Existence of the term $T(k,\mu) \ne 0$}
\label{sec:det}

In this section we establish the important detail that for each positive integer $k$ there is a number $\mu \in [0,m]$ such that $T(k,\mu) \ne 0$.

Recall that the function $T(k,\mu)$ is a linear combination of the terms $B_{k,0,j}(1)$, where $j=0,1,\ldots, m$. The idea is that if a determinant of a matrix whose $(j+1)$th row consists of the terms $B_{k,j,0}(1), B_{k,j,1}(1), \ldots, B_{k,j,m}(1)$ is non-zero, then all of the rows cannot be zero.

First we compute the determinant. From the work of Rhin and Toffin \cite{Rhin1986}, we get the following lemma:

\begin{lemma}\label{lemma:det} \cite[Proposition, p.\ 290]{Rhin1986}
Let
$$
\Delta (t) :=
\begin{vmatrix}
B_{k,0,0}(t) & B_{k,0,1}(t) & \cdots & B_{k,0,m}(t) \\
B_{k,1,0}(t) & B_{k,1,1}(t) & \cdots & B_{k,1,m}(t) \\
\vdots & \vdots & \ddots & \vdots \\
B_{k,m,0}(t) & B_{k,m,1}(t) & \cdots & B_{k,m,m}(t) \\
\end{vmatrix}.
$$
Then
$$
\Delta (t) = \pm \frac{(k!)^m}{(mk+m)!} \left( t^{\frac{m(m+1)}{2}} \alpha_1 \alpha_2 \cdots \alpha_m \prod_{1 \le i < j \le m} (\alpha_i - \alpha_j) \right)^{2k+1}.
$$
\end{lemma}

Using the previous formula for the determinant, the needed result follows: 
\begin{lemma}
\label{lemma:det2}
Let us assume that the numbers $\alpha_i$ are non-zero and pairwise distinct. For every $k$, there is $\mu \in \{0,1,\ldots,m\}$ such that 
$$
\lambda_0 B_{k,\mu,0}(1) + \lambda_1 B_{k,\mu,1}(1) + \ldots + \lambda_m B_{k,\mu,m}(1) \neq 0.
$$
Indeed, for every positive integer $k$ there must be a number $\mu$ such that $T(k,\mu) \ne 0$.
\end{lemma}

\begin{proof}
We see from Lemma \ref{lemma:det} that $\Delta (t)$ is non-zero for a non-zero $t$ when the numbers $\alpha_j$ are non-zero and pairwise different. In particular, the matrix defining $\Delta (1)$ is invertible, and since the coefficient vector $(\lambda_0, \lambda_1, \ldots, \lambda_m)$ is not the zero vector, one must have $\lambda_0 B_{k,\mu,0}(1) + \lambda_1 B_{k,\mu,1}(1) + \ldots + \lambda_m B_{k,\mu,m}(1) \neq 0$ for some $\mu \in \{0,1,\ldots,m\}$. By the definition of the term $T(k,\mu)$, this also implies that for every number $k$, there is a number $\mu$ such that $T(k,\mu) \ne 0$
\end{proof}

\section{Estimates for the polynomials and the remainders}
\label{sec:estimates}
\allowdisplaybreaks

To prove a lower bound for the linear form in logarithms $\Lambda_p$ in Section \ref{sec:proofs}, we need to have 
$$
|B_{k,\mu,0}(1) \Lambda_p|_p > | \lambda_1 S_{k,\mu,1}(1) + \ldots + \lambda_m S_{k,\mu,m}(1) |_p
$$
for $k$ large enough. This is shown via contradiction, and for that, we need upper bounds for the terms $B_{k,\mu,0}(1)$, $B_{k,\mu,j}(1)$, and $S_{k,\mu,j}(1)$, where $j=1,\ldots,m$. They are derived below.

%We define 
%$$
%g(x):=\frac{(1+x)(1+M+x)^m}{x},
%$$
%and
%$$
%\beta := \frac{1-m+\sqrt{(m+1)^2+4Mm}}{2m}.
%$$
%The proof of the next lemma follows the proof of \cite[Lemma 6]{Rhin1986}.

%Estimate for B0
\begin{lemma}
\label{lemma:B0}
Let $\alpha_1,\alpha_2,\ldots, \alpha_m$ be real numbers and let $\mu \in [0,m]$, $j \in [1,m]$, and $k>0$ be integers. Then 
$$
|B_{k,\mu,0}(1)| \le 2^{k+m-1}\cdot3^{mk}d_{mk+m}\cdot
\begin{cases}
1 & \text{ if } M< 1; \\
M^{mk} & \text{ if } M> 1.
\end{cases}
$$
\end{lemma}

\begin{proof}
According to to definition \eqref{def:B0} of the term $B_{k,\mu,0}(1)$, we have
\begin{align*}
\left|B_{k,\mu,0}(1)\right|&=d_{mk+\mu} \left|A_{k,\mu,0}(-1)\right| \le d_{mk+\mu}\sum_{i=0}^{mk} \binom{i+k+\mu}{k} \left|\sigma_i\left(k, \overline{\alpha}\right)\right|.
\end{align*}
Let us first estimate the term $\binom{i+k+\mu}{k}$. Keeping in mind that $\sum_{h=0}^n \binom{n}{h}=2^n$ and $\binom{n}{h}=\binom{n}{n-h}$, we can write
$$
\binom{i+k+\mu}{k} \le 2^{i+k+\mu-1}
$$
if $i+k+\mu \ne 2k$, since the sum $\sum_{h=0}^n \binom{n}{h}$ can be divided into two parts and the term $\binom{i+k+\mu}{k}$ is in the part which is 
at most as large as the other one. Further, since $\binom{2k}{k} \le 2\binom{2k}{k-1}$, we can estimate
$$
2\binom{i+k+\mu}{k} = 2^{i+k+\mu}-\sum_{\substack{h=0 \\ h\ne k-1, k, k+1}}^{i+k+\mu} \binom{i+k+\mu}{h} \le 2^{i+k+\mu},
$$
and thus the wanted estimate holds also when $i+k+\mu = 2k$.

We have obtained
$$
\left|B_{k,\mu,0}(1)\right| \le 2^{k+\mu-1}d_{mk+\mu}\sum_{i=0}^{mk} 2^i\left|\sigma_i\left(k, \overline{\alpha}\right)\right|.
$$
By \Cref{sigmalemma}, the right-hand side is
$$
\le 2^{k+\mu-1}d_{mk+\mu}\prod_{i=1}^{m} \left(|\alpha_i|+2\right)^k \le 2^{k+\mu-1}d_{mk+\mu}\prod_{i=1}^{m} \left(M+2\right)^k .
$$
The claim follows from the estimates $M+2< 3$ for $M < 1$ and $M+2 \le 3M$ for $M>1$, and the fact $\mu \le m$.
\end{proof}

We are ready with the term $B_{k,\mu,0}(1)$ and move on to estimate the term $B_{k,\mu,j}(1)$ for $j \in \{1,\ldots,m\}$.

%Estimate for Bj
\begin{lemma}\label{lemBj}
Assume $k \ge 1$, $\mu \in [0,m]$, and $j \in [1,m]$ are integers. We have
$$
|B_{k,\mu,j}(1)| \le 2^{k+m-1}\cdot 3^{mk} d_{mk+m}\cdot
 \frac{M^{mk+m+1}-M}{M-1},
$$
keeping in mind that $M \ne 1$.
\end{lemma}

\begin{proof}
This proof is quite similar to the proof of \Cref{lemma:B0}: First we use the definition of the term $B_{k,\mu,j}(1)$, then estimate a certain binomial coefficient and the term $\sigma_i$ to establish the claim.

By definition \eqref{def:Bj} of the term $B_{k,\mu,j}(1)$, we have
\begin{equation}\label{eq:B0First}
\begin{split}
&|B_{k,\mu,j}(1)| \\
= \;&\left|\alpha_jd_{mk+\mu} \sum_{N=0}^{mk+\mu-1} r_{N,j} (-1)^N \right| \\
\le \;&\left|\alpha_jd_{mk+\mu}\right| \sum_{N=0}^{mk+\mu-1}  \sum_{i=mk-\min\{N,mk\}}^{mk} \binom{i+k+\mu}{k} \cdot \frac{\left|\sigma_{i}\left(k, \overline{\alpha}\right)\right|\left| \alpha_j^{N-mk+i}\right|}{\left|N-mk+i+1\right|}.
%&\le d_{mk+\mu} \sum_{N=0}^{mk+\mu-1} \sum_{i=mk-\min\{N,mk\}}^{mk} \binom{i+k+\mu}{k} \frac{M^{N+1}}{N-mk+i+1} \\
%&\; \cdot\sum_{i_1 + \ldots + i_m = i} \binom{k}{i_1} \cdots \binom{k}{i_m}.
\end{split}
\end{equation}

As was mentioned in the proof of \Cref{lemma:B0}, it holds
$$
\binom{i+k+\mu}{k} \le 2^{i+k+\mu-1}.
$$
Further, since $N-mk+i \ge 0$, the right-hand side of \eqref{eq:B0First} is
\begin{align*}
& \le 2^{k+\mu-1}Md_{mk+\mu} \sum_{N=0}^{mk+\mu-1} M^{N-mk} \sum_{i=0}^{mk} \left|\sigma_{i}\left(k, \overline{\alpha}\right)\right|\left( 2M\right)^i.
\end{align*}
By \Cref{sigmalemma}, this can be further estimated as
\begin{align*}
& \le 2^{k+\mu-1}Md_{mk+\mu} \sum_{N=0}^{mk+\mu-1} M^{N-mk} \prod_{i=0}^{m} \left(|\alpha_i|+2M\right)^k \\
&\le 2^{k+\mu-1}\cdot3^{mk}Md_{mk+\mu} \sum_{N=0}^{mk+\mu-1} M^{N}.
\end{align*}

Recall that $M \ne 1$ since $\alpha<1$. As
$$
 \sum_{N=0}^{mk+\mu-1} M^{N}=\frac{M^{mk+\mu}-1}{M-1} 
$$
and $\mu \le m$, the claim follows.

%Keeping in mind that $\sum_{h=0}^n \binom{n}{h}=2^n$, we have
%$$
%\sum_{i_1 + \ldots + i_m = i} \binom{k}{i_1} \cdots \binom{k}{i_m} \le 2^{mk}
%$$
%and since $\mu \le m$, we also have
%{\color{blue} Tahan voisi kayttaa myos Stirlingin kaavasta saatavaa arviota, mutta tasta saadaan natimpi muotoilu.}
%$$
%\binom{i+k+\mu}{k} \le 2^{i+k+\mu}\le 2^{i+k+m}.
%$$
%Further, noting that $N-mk+i+1 \ge 1$ and assuming $M \ne 1$, we have obtained
%\begin{align*}
%|B_{k,\mu,j}(1)| &\le d_{mk+\mu} \sum_{N=0}^{mk+\mu-1} \sum_{i=mk-\min\{N,mk\}}^{mk} 2^{i+mk+k+m} M^{N+1} \\
%& \le d_{mk+\mu} M \frac{M^{mk+\mu}-1}{M-1} 2^{k(m+1)+m}\left(2^{mk+1}-1\right) \\
%& < d_{mk+\mu} \frac{M^{mk+\mu+1}-M}{M-1} 2^{k(2m+1)+m}.
%\end{align*}
%Since $\mu \le m$, the claim for $M \ne 1$ follows.
%Similarly, for $M=1$ we get
%\begin{align*}
%|B_{k,\mu,j}(1)| &< d_{mk+m} m(k+1)2^{k(2m+1)+m}.
%\end{align*}
\end{proof}

%Estimate for S
\begin{lemma}\label{lemSj}
Assume $k \ge 1$, $\mu \in [0,m]$, and $j \in [1,m]$ are integers.We have
$$
|S_{k,\mu,j}(1)|_p \le \left(mk+k+1\right) \alpha^{mk+k+1}.
$$
\end{lemma}

\begin{proof}
We use representation \eqref{eq:remainder}. First we note that for all integers $i \in [mk-\min\{N, mk\}, mk]$, we have
\begin{align*}
\left| \sigma_i \left(k, \overline{\alpha} \right) \right|_p &= \left| (-1)^i \sum_{i_1 + \ldots + i_m = i} \binom{k}{i_1} \cdots \binom{k}{i_m} \cdot \alpha_1^{k-i_1} \cdots \alpha_m^{k-i_m} \right|_p \\
&\le \max_{i_1 + \ldots + i_m = i} \left| \binom{k}{i_1} \cdots \binom{k}{i_m} \cdot \alpha_1^{k-i_1} \cdots \alpha_m^{k-i_m} \right|_p \\
&\le \alpha^{mk-i}
\end{align*}
and
$$
\frac{1}{|N-mk+i+1|_p} \le  p^{\log_p(N-mk+i+1)} \le N-mk+i+1.
$$

Using the previous estimates,
\begin{align*}
|S_{k,\mu,j}(1)|_p &= \left| d_{mk+\mu}\alpha_j \sum_{N=mk+k+\mu}^\infty r_{N,j}(-1)^N \right|_p \\
&\le |\alpha_j|_p \max_{N \ge mk+k+\mu} |r_{N,j}|_p \\
&= |\alpha_j|_p \max_{N \ge mk+k+\mu} \left| \sum_{i=0}^{mk} (-1)^{mk} \binom{i+k+\mu}{k} \frac{\sigma_{i}\left(k, \overline{\alpha}\right) \alpha_j^{N-mk+i}}{N-mk+i+1} \right|_p \\
&\le |\alpha_j|_p \max_{N \ge mk+k+\mu} \left\{ \max_{0 \le i \le mk} \left| \binom{i+k+\mu}{k} \frac{\sigma_{i}\left(k, \overline{\alpha}\right) \alpha_j^{N-mk+i}}{N-mk+i+1} \right|_p \right\} \\
&\le |\alpha_j|_p \max_{N \ge mk+k+\mu} \left\{ \max_{0 \le i \le mk} \left\{ \left(N-mk+i+1\right)\alpha^{mk-i} \cdot \alpha^{N-mk+i} \right\} \right\} \\
&\le \alpha \max_{N \ge mk+k+\mu} \left\{ (N+1) \alpha^N \right\}.
\end{align*}

The function $(N+1)\alpha^N$ is decreasing for $N \ge-\frac{1}{\log\alpha}-1$. Further, since $\alpha<1$, we also have $\alpha \le \frac{1}{p}$, and hence 
$$
-\frac{1}{\log\alpha}-1 \le \frac{1}{\log p}-1 \le \frac{1}{\log 2}-1<1.
$$
Thus, the maximum value for $(N+1) \alpha^N$ is obtained at $N=mk+k+\mu$. Keeping in ming that $\mu \ge 0$, we get the claim.
\end{proof}

%---------
%Kokeillaan, milta asia nayttaa alkuperaisella $s_{k,\mu,j}$
%
%%Estimate for s_{k,\mu,j} 
%\begin{lemma}
%Let??? $s_{k,\mu,j} $
%\end{lemma}
%
%\begin{proof}
%By the definition of the term $s_{k,\mu,j}$ and ??? we have
%\begin{equation*}
%\begin{aligned}
%s_{k,\mu,j}&=(mk+\mu)! \cdot \alpha_j \sum_{N=mk+k+\mu}^\infty  \sum_{i=0}^{mk} (-1)^{mk+N} k! \binom{i+k+\mu}{k}\cdot \\
%&\quad\cdot \frac{\alpha_j^{N-mk+i}(-1)^i \sum_{i_1 + \ldots + i_m = i} \binom{k}{i_1} \cdots \binom{k}{i_m} \cdot \alpha_1^{k-i_1} \cdots \alpha_m^{k-i_m}}{N-mk+i+1}.
%\end{aligned}
%\end{equation*}
%\end{proof}

%%MERKINNOISTA
%\subsection{Specifying the numbers $\alpha_j$}
%\label{sec:newalphas}
%
%In \Cref{thm:mainresult} we assume that the numbers $\alpha_j$ are integers and powers of some prime number $p$. Hence, In order to settle certain problems that would otherwise arise in the estimates of the next section, we shall now choose $\alpha_j = p^{a_j}$, where the exponents $a_j$ are positive integers. Then
%$$
%M = \max_{1 \le j \le m} |\alpha_j| = p^{\max_{1 \le j \le m} \{a_j\}} =: p^A
%$$
%and
%$$
%\alpha = \max_{1 \le j \le m} |\alpha_j|_p = p^{-\min_{1 \le j \le m} \{a_j\}} =: \frac{1}{p^a}.
%$$
%Please notice that the numbers $a$ and $A$ are non-negative integers.

%LOWER BOUND
%\section{Lower bound}
%\label{sec:low}

%CONTRADICTIONS
\section{Computing $|B_{k,\mu,0}(1) \Lambda_p|_p$}
\label{sec:Contradiction}

As was described in \Cref{sec:generalidea}, we shall show that $|Q^{mk+m}B_{k,\mu,0}(1) \Lambda|_p=\left|T(k,\mu)\right|_p$. Since we have
\begin{equation*}
T(k, \mu) = Q^{mk+m}B_{k,\mu,0}(1)\Lambda_p + \sum\limits_{j=1}^m \lambda_jQ^{mk+m}S_{k,\mu,j}(1),
\end{equation*}
from the properties of the $p$-adic absolute value it follows that it suffices to show that we have 
$$
|B_{0,k,\mu}(1) \Lambda|_p > | \lambda_1 S_{k,\mu,1}(1) + \ldots + \lambda_m S_{k,\mu,m}(1) |_p.
$$
This is done via contradiction; namely, if
\begin{equation}
\label{eq:BSmaller}
|B_{0,k,\mu}(1) \Lambda|_p \le | \lambda_1 S_{k,\mu,1}(1) + \ldots + \lambda_m S_{k,\mu,m}(1) |_p,
\end{equation}
where $T(k,\mu) \ne 0$, then a certain product has to be greater than one and smaller than one at the same time. Hence inequality \eqref{eq:BSmaller} is impossible. The result is needed in \Cref{sec:proofs} to prove the main theorems.

The proof in this section contains two main steps: First, in \Cref{lemma:upperbound1}, we show that from inequality \eqref{eq:BSmaller}, it follows that a certain product has to be greater than one. Then, in \Cref{lemma:contr2,lemma:contr3}, the contradiction is derived.

On several occasions, estimates for terms of the form $xe^x$ will be needed. A branch $W_{-1}(x)$ of the Lambert $W$ function will be used which is defined for all real numbers $x \in \left[-\frac{1}{e},0\right)$ and for which we have
$$
W_{-1}\left(-\frac{1}{e}\right)=-1 \quad\text{and}\quad \lim_{x \to 0^-}W_{-1}(x)=-\infty.
$$

%Preliminaries for the contradiction
\subsection{Preliminaries for the term $|B_{k,\mu,0}(1) \Lambda_p|_p$}

Before going to the actual proofs, we shortly mention a lemma that is needed in the coming estimates to deal with the term $d_{mk+\mu} = \lcm (1,2,\ldots,mk+\mu)$.

%Estimate for d
\begin{lemma}\cite[Theorem 12]{Rosser1962}
\label{lemma:dEst}
Let $n$ be a positive integer and $n \le x$. Then
$$
\log (\lcm (1,2,\ldots,n)) < 1.03883x.
$$
\end{lemma}

Now we are ready to prove the first step where we apply the fact that for each positive integer $k$ there exists a number $\mu$ such that $T(k,\mu) \ne 0$ (\Cref{lemma:det2}).

%Step 1 for the contradiction
\begin{lemma}
\label{lemma:upperbound1}
Let $\lambda_1,\lambda_2, \ldots, \lambda_m$ be integers and $\alpha_1,\alpha_2,\ldots, \alpha_m$ non-zero and pairwise distinct rational numbers with $\alpha<1$. Let $k \ge 1$ be an integer and $\mu \in [0,m]$ be an integer such that $T(k,\mu) \ne 0$. Assume 
$$
|B_{0,k,\mu}(1) \Lambda|_p \le | \lambda_1 S_{k,\mu,1}(1) + \ldots + \lambda_m S_{k,\mu,m}(1) |_p.
$$  
Then
\begin{equation}\label{eq:1les}
\begin{split}
1 < &\; e^{1.03883(mk+m)}Q^{mk+m}H(m+1) 2^{k+m-1} 3^{mk} \cdot \frac{M^{mk+m+1}-M}{M-1}\\
& \cdot \left(mk+k+1\right) \alpha^{mk+k+1} =: \Omega_1
\end{split}
\end{equation}
when $M>1$, and
\begin{equation}
\label{eq:omega2}
\begin{split}
1 < \;&e^{1.03883(mk+m)}Q^{mk+m}H(m+1) 2^{k+m-1} 3^{mk}(mk+m) \\
&\cdot \left(mk+k+1\right) \alpha^{mk+k+1} =: \Omega_2
\end{split}
\end{equation}
when $M < 1$.
\end{lemma}

\begin{proof}
The claim follows by using the product formula. Since $T(k,\mu) \neq 0$, we get, as described in \Cref{sec:generalidea},
\begin{align*}
1 &\le \left( \sum_{i=0}^m |\lambda_i| \right) Q^{mk+m}\max_{0 \le j \le m} \{ |B_{k,\mu,j}(1)| \} \cdot  \max_{1 \le j \le m} \{|S_{k,\mu,j}(1)|_p\}.
\end{align*}
Let us bound the terms on the right-hand side of the previous inequality.

First, by the definition of the number $H$, we have $\sum_{i=0}^m |\lambda_i| \le (m+1)H$. As for the terms $B_{k,\mu,j}(1)$, we combine the information from Lemmas \ref{lemma:B0} and \ref{lemBj}. When $M > 1$, we have
$$
M^{mk} < \frac{M^{mk+m+1}-M}{M-1},
$$
so
\begin{equation}\label{eq:estmaxB}
\max_{0 \le j \le m} \{ |B_{k,\mu,j}(1)| \} \le 2^{k+m-1} 3^{mk} d_{mk+m} \cdot \frac{M^{mk+m+1}-M}{M-1} \quad \text{for } M>1.
\end{equation}
When $M < 1$, it holds
$$
\max\left\{1,\frac{M^{mk+m+1}-M}{M-1}\right\} \le \max\left\{1,M(mk+m)\right\} < mk+m.
$$
Hence, we have
\begin{equation}\label{eq:maxBMle1}
\max_{0 \le j \le m} \{ |B_{k,\mu,j}(1)| \} < 2^{k+m-1} 3^{mk} d_{mk+m}(mk+m) \quad \text{for } M < 1.
\end{equation}
%Lastly, it is easy to see that the previous estimate holds also for $M=1$.

Finally, from Lemma \ref{lemSj}, we get
$$
\max_{1 \le j \le m} \left\{ |S_{k,\mu,j}(1)|_p \right\} \le \left(mk+k+1\right) \alpha^{mk+k+1}.
$$

The claim follows from the previous estimates and \Cref{lemma:dEst}.
\end{proof}

%Contradiction, M>1
\subsection{Computing $|B_{k,\mu,0}(1) \Lambda_p|_p$ when $M>1$}
Next we prove that inequality \eqref{eq:BSmaller} cannot hold when $M>1$. We need the following small lemma concerning the Lambert $W$ function:
\begin{lemma}\label{lemma:W1First}
Let $x<0$ and $y$ be real numbers. If $y \ge -\frac{1}{e}$ and $x < W_{-1}(y)$,
then we have
$
xe^x > y.
$
%If $y<-\frac{1}{e}$, then $xe^x > y$ is true for all $x$.
\end{lemma}

\begin{proof}
Because the function $W_{-1}(x)$ is an inverse of the function $xe^x$ and is decreasing for $x \in \left[-\frac{1}{e}, 0\right)$, the claim follows.
\end{proof}

We move on to prove the case $M>1$:

\begin{lemma}
\label{lemma:contr2}
Let the numbers $\alpha_j$, $\lambda_j$, $p$, and $M>1$ be given as in \Cref{thm:mainresult2}. Assume that $T(k,\mu) \ne 0$ and that conditions \eqref{eq:assumpless1} and \eqref{eq:condW1} hold. Let $k$ be an integer such that
\begin{equation}
\label{eq:kassump}
\begin{split}
k \ge \;& \max\left\{1, \frac{1}{\log \left( 2\cdot 3^me^{1.03883m} Q^m M^m \alpha^{m+1} \right)} \vphantom{W_{-1}\left(\frac{\log \left( 2\cdot 3^me^{1.03883m}  M^m \alpha^{m+1} \right)}{2^{m-1} e^{1.03883m}Q^m H(m+1) \cdot \frac{M^{m+1}}{M-1} \cdot (m+2) \alpha}\right)}\right. \\
&\left. \cdot W_{-1}\left(\frac{\log \left( 2\cdot 3^me^{1.03883m}  Q^mM^m \alpha^{m+1} \right)}{2^{m-1}e^{1.03883m}Q^m H(m+1)\cdot \frac{M^{m+1}}{M-1} \cdot (m+2) \alpha}\right) \right\}.
\end{split}
\end{equation}

Then we have
$$
|B_{k,\mu,0}(1) \Lambda|_p > | \lambda_1 S_{k,\mu,1}(1) + \ldots + \lambda_m S_{k,\mu,m}(1) |_p.
$$
\end{lemma}

\begin{proof}
%We advance as described \Cref{sec:generalidea}, showing that the term $\Omega_1$ from \eqref{eq:1les} has to be greater and smaller than one at the same time.

Suppose that 
$$
|B_{k,\mu,0}(1) \Lambda|_p \le | \lambda_1 S_{k,\mu,1}(1) + \ldots + \lambda_m S_{k,\mu,m}(1) |_p.
$$ 
By \Cref{lemma:upperbound1}, we have $1 < \Omega_1$. Our goal is to show that $\Omega_1<1$ which gives the wanted contradiction.

Since $k \ge 1$, it holds
\begin{align*}
\Omega_1 = \;&e^{1.03883(mk+m)}Q^{mk+m}H(m+1) 2^{k+m-1} 3^{mk} \cdot \frac{M^{mk+m+1}-M}{M-1} \cdot \left(mk+k+1\right) \alpha^{mk+k+1} \\
< \;&2^{m-1} e^{1.03883m}Q^m H(m+1) \cdot \frac{M^{m+1}}{M-1} \cdot (m+2) \alpha \cdot k \left( 2\cdot 3^me^{1.03883m}Q^mM^m \alpha^{m+1} \right)^k \\
= \;&\frac{2^{m-1}e^{1.03883m}Q^m H(m+1)\cdot \frac{M^{m+1}}{M-1} \cdot (m+2) \alpha}{\log \left( 2\cdot 3^me^{1.03883m}Q^m  M^m \alpha^{m+1} \right)} \\
& \cdot \log \left( 2\cdot 3^me^{1.03883m}Q^m  M^m \alpha^{m+1} \right) k \left( 2\cdot 3^me^{1.03883m}Q^m  M^m \alpha^{m+1} \right)^k.
\end{align*}
Due to assumption \eqref{eq:kassump} and Lemma \ref{lemma:W1First}, the above expression is less than 1, giving the wanted contradiction.
\end{proof}

\subsection{Computing $|B_{k,\mu,0}(1) \Lambda_p|_p$ when $M<1$}

The idea of the proof in the case $M<1$ is exactly the same as in the previous section, just with a bit more complicated formula. Again we need a property of the Lambert $W$ function:

\begin{lemma}
\label{lemma:W1Second}
Let $x>0$ and $y>0$ be real numbers. If $y \le \frac{4}{e^2}$ and $x \ge -2W_{-1}\left(-\frac{\sqrt{y}}{2}\right)$, then we have
$
x^2e^{-x}\le y.
$
%If $y> 4/e^2$, then we have $x^2e^{-x}\le y$ for all $x>0$.
\end{lemma}

\begin{proof}
Let us
%first consider the case $y<4/e^2$ and 
solve the inequality $x^2e^{-x}\le y$. We can write it in the form 
$$
xe^{-x/2}\le \sqrt{y} \quad\iff \quad -\frac{x}{2}e^{-x/2} \ge -\frac{\sqrt{y}}{2}
$$
since $x,y>0$. The branch $W_{-1}(z)$ is decreasing for all $-\frac{1}{e}\le z<0$. Hence we get $-\frac{x}{2} \le W_{-1}\left(-\frac{\sqrt{y}}{2}\right)$, and the claim follows.

%Further, taking the derivative of the function $x^2e^{-x}$ we see that it is always at most $4/e^2$. This proves the second case.
\end{proof}

Finally, we can continue to the case $M<1$:
\begin{lemma}
\label{lemma:contr3}
Let the numbers $\alpha_j$, $\lambda_j$, $p$, and $M<1$ be given as in \Cref{thm:mainresult3}. Assume that $T(k,\mu) \ne 0$ and that conditions \eqref{eq:assumpless1} and \eqref{eq:condH5} hold.
%If
%\begin{equation}\label{eq:condH4}
%\frac{\left( \log \left( 2\cdot 3^me^{1.03883m}Q^mM^m \alpha^{m+1} \right) \right)^2}{2^m e^{1.03883m}Q^m Hm(m+1)(m+2) \alpha} \le \frac{4}{e^2},
%\end{equation}
%then 
Let $k$ be an integer such that
\begin{multline}\label{eq:kassump2}
k \ge \max \left\{ 1, -2W_{-1}\left( -\frac{1}{2} \sqrt{\frac{\left( \log \left( 2\cdot 3^me^{1.03883m}Q^mM^m \alpha^{m+1} \right) \right)^2}{2^m e^{1.03883m}Q^m H(m+1) \cdot m(m+2) \alpha}} \right) \right. \\
\left. \vphantom{\left( -\frac{1}{2} \sqrt{\frac{\left( \log \left( 2\cdot 3^me^{1.03883m}Q^mM^m \alpha^{m+1} \right) \right)^2}{2^m e^{1.03883m}Q^m H(m+1) \cdot m(m+2) \alpha}} \right)}
\cdot \left(\log \left( 2^{-1}\cdot 3^{-m}e^{-1.03883m}Q^{-m}M^{-m} \alpha^{-(m+1)} \right) \right)^{-1} \right\}.
\end{multline}
%and let $k \ge 1$ otherwise.

Then we have
$$
|B_{k,\mu,0}(1) \Lambda|_p > | \lambda_1 S_{k,\mu,1}(1) + \ldots + \lambda_m S_{k,\mu,m}(1) |_p.
$$
\end{lemma}

\begin{proof}
We use similar ideas as in the proof of \Cref{lemma:contr2}. Suppose that 
$$
|B_{k,\mu,0}(1) \Lambda|_p \le | \lambda_1 S_{k,\mu,1}(1) + \ldots + \lambda_m S_{k,\mu,m}(1) |_p.
$$ 
Then, by \Cref{lemma:upperbound1}, we have $1 < \Omega_2$, where $\Omega_2$ is given as in formula \eqref{eq:omega2}. Our goal is to show $\Omega_2<1$.

Since $k \ge 1$, we have
\begin{align*}
\Omega_2 = \;&e^{1.03883(mk+m)}Q^{mk+m}H(m+1) 2^{k+m-1} 3^{mk} (mk+m) \cdot \left(mk+k+1\right) \alpha^{mk+k+1} \\
< \;&2^{m-1} e^{1.03883m}Q^m H(m+1) \cdot m(m+2) \alpha \cdot k(k+1) \cdot \left( 2\cdot 3^me^{1.03883m}Q^mM^m \alpha^{m+1} \right)^k \\
\le \;&2^{m-1} e^{1.03883m}Q^m H(m+1) \cdot m(m+2) \alpha \cdot 2k^2 e^{k \log \left( 2\cdot 3^me^{1.03883m}Q^mM^m \alpha^{m+1} \right)}\\
= \;&\frac{2^m e^{1.03883m}Q^m Hm(m+1)(m+2) \alpha}{\left( \log \left( 2\cdot 3^me^{1.03883m}Q^mM^m \alpha^{m+1} \right) \right)^2} \\
&\cdot \left( k \log \left( 2\cdot 3^me^{1.03883m}Q^mM^m \alpha^{m+1} \right) \right)^2 e^{k \log \left( 2\cdot 3^me^{1.03883m}Q^mM^m \alpha^{m+1} \right)}.
\end{align*}
Again, assumption \eqref{eq:kassump2} and Lemma \ref{lemma:W1Second} together ensure that the above expression is less than 1, giving a contradiction.
\end{proof}

%PROOF OF THE MAIN RESULTS
\section{Proofs for the main results}
\label{sec:proofs}

In this section, we shall prove the main results and corollaries mentioned in \Cref{sec:results}. The following lemma is needed to bound the branch of the Lambert $W$ function that we use:

\begin{lemma}\cite[Theorem 3.2]{Alzahrani2018}
\label{lemma:estLambertW}
Let $t$ be a non-negative real number. Then
$$
-\log(t+1)-t-2+\log(e-1)<W_{-1}\left(-e^{-t-1}\right)\le-\log(t+1)-t-1.
$$
\end{lemma}

% \vspace{0.2cm}
% \hrule
% \vspace{0.2cm}

% \textcolor{blue}{$z$-funktio-vaihtoehto}

% \begin{lemma}
% If $y \ge se^s$ and $s \ge e$, for the inverse function $z(y)$ of the function $Y(z)=z \log z$ it holds
% $$
% z(y) \le \left( 1 + \frac{\log s}{s} \right) \frac{y}{\log y}.
% $$
% \end{lemma}

% \vspace{0.2cm}
% \hrule
% \vspace{0.2cm}

%Proof of the case M>1
\subsection{Proofs of the case $M>1$}

In this section, we prove the results mentioned in \Cref{sec:resultsMgreater}. We start with the main result which provides an explicit estimate for the linear form $\Lambda_p$ for all integers $H$.

\subsubsection{Proof of the main theorem}

%Main result
\begin{proof}[Proof of \Cref{thm:mainresult2}]
As was described in \Cref{sec:generalidea}, it is sufficient to prove an upper bound for the function $T(k,\mu)$.

By the definition of the number $T(k,\mu)$ and estimate \eqref{eq:estmaxB}, we get
\begin{equation}
\label{eq:estW2}
\begin{split}
|T(k,\mu)| &\le Q^{mk+m}\left( \sum_{j=0}^m |\lambda_j| \right) \max_{0 \le j \le m} \left| B_{k,\mu,j}(1)\right| \\
&\le e^{1.03883m}Q^{mk+m}H(m+1)2^{k+m-1}3^{mk} \cdot
\frac{M^{mk+m+1}-M}{M-1} \\
& < e^{1.03883m}Q^mH(m+1) 2^{m-1} \cdot \frac{M^{m+1}}{M-1} \cdot \left( 2\cdot3^m e^{1.03883m}Q^m M^m \right)^k \\
& = e^{1.03883m}Q^mH(m+1) 2^{m-1} \cdot \frac{M^{m+1}}{M-1} \cdot e^{k \log \left( 2\cdot3^m e^{1.03883m} Q^mM^m \right)}.
\end{split}
\end{equation}
We take a closer look at the terms containing $k$ on the last line of \eqref{eq:estW2}.
%We divide the consideration to two cases depending on whether inequality \eqref{eq:condW1}, which essentially tells how big $k$ we have to choose, holds or not.

%First we assume that inequality \eqref{eq:condW1} does not hold. Hence we can set $k=1$. 
%{\color{blue} Riitt\"a\"ak\"o sittenkin, ett\"a $k\ge 1$ eli voidaanko valita $k=1$?}
%According to the the right-hand side of \eqref{eq:estW2}, we have obtained
%\begin{equation*}
%    \left|T(k,\mu)\right| < \frac{H^{\omega(m, M, Q, \alpha)}}{c(m, M, Q, \alpha)},
%\end{equation*}
%where
%\begin{equation}\label{eq:cmMQ}
 %   c(m, M, Q, \alpha)= \left(2^{m+1}3^{2m}(m+1)e^{3.11649m }Q^{3m} \frac{M^{3m+1}}{M-1}\right)^{-1}
%\end{equation}
%{\color{blue} Jos $k=1$ k\"ay, niin voidaan kirjoittaa}
%\begin{equation}
%\label{eq:cmMQ}
%    c(m, M, Q, \alpha)= \left(2^{m}3^{m}(m+1)e^{2.07766m}Q^{2m} \frac{M^{2m+1}}{M-1}\right)^{-1}
%\end{equation}
%and 
%\begin{equation}\label{eq:omegamMQ}
 %   \omega_1(m, M, Q, \alpha)=1.
%\end{equation}
%This proves the case when inequality \eqref{eq:condW1} does not hold. Please notice that we did not need assumption \eqref{eq:condH1} for the number $H$.

%Let us now move on to the case where inequality \eqref{eq:condW1} holds.
Following condition \eqref{eq:kassump}, we choose 
\begin{equation*}
\begin{split}
k = \;& \left\lceil \frac{1}{\log \left( 2\cdot 3^me^{1.03883m}Q^m  M^m \alpha^{m+1} \right)} \vphantom{W_{-1}\left(\frac{\log \left( 2\cdot 3^me^{1.03883m}  M^m \alpha^{m+1} \right)}{2^{m-1}e^{1.03883m} H(m+1)\cdot \frac{M^{m+1}}{M-1} \cdot (m+2) \alpha}\right) }\right. \\
&\cdot \left. W_{-1}\left(\frac{\log \left( 2\cdot 3^me^{1.03883m}  Q^mM^m \alpha^{m+1} \right)}{2^{m-1}e^{1.03883m}Q^m H(m+1)\cdot \frac{M^{m+1}}{M-1} \cdot (m+2) \alpha}\right) \right\rceil. 
\end{split}
\end{equation*}
Taking the ceiling function of a positive term, we have $k \ge 1$. Let us bound $k$ from above by considering the function inside the ceiling function.
%Then
%\begin{equation*}
%\begin{split}
%k \le \;&\frac{1}{\log \left( e^{2.07766m} 2^{2m+1} (m+1) M^m \alpha^{m+1} \right)} \\
%&\cdot W_{-1}\left(\frac{\log \left( e^{2.07766m} 2^{2m+1} (m+1) M^m \alpha^{m+1} \right)}{e^{2.07766m} H(m+1)3^m \cdot \frac{M^{m+1}}{M-1} \cdot (m+2) \alpha}\right) +1.
%\end{split}
%\end{equation*}

We use Lemma \ref{lemma:estLambertW} to estimate the term 
$$
W_{-1}\left(\frac{\log \left( 2\cdot 3^me^{1.03883m}Q^m  M^m \alpha^{m+1} \right)}{2^{m-1}e^{1.03883m}Q^m H(m+1)\cdot \frac{M^{m+1}}{M-1} \cdot (m+2) \alpha}\right).
$$
First, we have
\begin{align*}
& \frac{\log \left( 2\cdot 3^me^{1.03883m}Q^m  M^m \alpha^{m+1} \right)}{2^{m-1}e^{1.03883m}Q^m H(m+1)\cdot \frac{M^{m+1}}{M-1} \cdot (m+2) \alpha} \\
= \;& \frac{-\log \left( 2^{-1}\cdot 3^{-m}e^{-1.03883m}  Q^{-m}M^{-m} \alpha^{-m-1} \right)}{e^{1.03883m +m\log Q+ \log H + \log(m+1) + (m-1)\log 2 + (m+1)\log M - \log(M-1) + \log(m+2) + \log \alpha}} \\
=\; &-e^{-t-1},
\end{align*}
where 
\begin{align*}
t = \;&1.03883m +m\log Q+ \log H + \log(m+1) + (m-1)\log 2 + (m+1)\log M \\
&- \log(M-1) + \log(m+2) + \log \alpha -\log \log \left( 2^{-1}\cdot 3^{-m}e^{-1.03883m}  Q^{-m} M^{-m} \alpha^{-m-1} \right)-1 \\
=\;&\log H+R_1(m,M,Q,\alpha)-1
\end{align*}
and where $R_1(m,M,Q,\alpha)$ is defined as in \eqref{def:R}. Because of assumption \eqref{eq:condW1}, we have $t>0$.

Hence, by Lemma \ref{lemma:estLambertW}, we can estimate 
\begin{align*}
&\left\lceil \frac{1}{\log \left( 2\cdot 3^me^{1.03883m}  Q^mM^m \alpha^{m+1} \right)} \vphantom{W_{-1}\left(\frac{\log \left( 2\cdot 3^me^{1.03883m}  Q^mM^m \alpha^{m+1} \right)}{2^{m-1}e^{1.03883m}Q^m H(m+1)\cdot \frac{M^{m+1}}{M-1} \cdot (m+2) \alpha}\right)}\right. \\
&\left. \cdot W_{-1}\left(\frac{\log \left( 2\cdot 3^me^{1.03883m}  Q^mM^m \alpha^{m+1} \right)}{2^{m-1}e^{1.03883m}Q^m H(m+1)\cdot \frac{M^{m+1}}{M-1} \cdot (m+2) \alpha}\right) \right\rceil \\
< \;&\frac{-\log(t+1)-t-2+\log(e-1)}{\log \left( 2\cdot 3^me^{1.03883m}  Q^mM^m \alpha^{m+1} \right)}+1 \\
< \;&\frac{\log(t+1)+t+2}{\log \left( 2^{-1}\cdot 3^{-m}e^{-1.03883m}  Q^mM^{-m} \alpha^{-m-1} \right)}+1.
\end{align*}
%We set the number $k$ be as in the right-hand side of the previous inequality. 

Now, the exponent of $e$ on the last line of estimate \eqref{eq:estW2} can be bounded as
\begin{align*}
&k \log \left( 2\cdot 3^m e^{1.03883m} Q^m M^m \right) \\
< \;& \left(\frac{\log\left(\log H+R_1(m,M,Q,\alpha)\right)+\log H+R_1(m,M,Q,\alpha)+1}{\log \left( 2^{-1}\cdot 3^{-m}e^{-1.03883m}  Q^mM^{-m} \alpha^{-m-1} \right)}+1\right) \\
&\cdot\log\left( 2\cdot 3^m e^{1.03883m} Q^m M^m \right) \\
=\;&(\log H) \cdot \frac{\log\left( 2\cdot 3^m e^{1.03883m} Q^m M^m \right)}{\log f\left(m,M,Q,\alpha \right)} \left(1+\frac{\log\left(\log H+R_1(m,M,Q,\alpha)\right)}{\log H}\right) \\
&+\left(\frac{R_1(m,M,Q,\alpha)+1}{\log f\left(m,M,Q,\alpha \right)}+1\right)\log\left( 2\cdot 3^m e^{1.03883m} Q^m M^m \right) \\
=\;&(\log H) \left(-\frac{(m+1) \log\alpha}{\log f(m,M,Q,\alpha)}-1\right)\left(1+\frac{\log\left(\log H+R_1(m,M,Q,\alpha)\right)}{\log H}\right) \\
&+\left(\frac{R_1(m,M,Q,\alpha)+1}{\log f(m,M,Q,\alpha)}+1\right)\log\left( 2\cdot 3^m e^{1.03883m} Q^m M^m \right),
\end{align*}
where $f(m,M,Q,\alpha)$ is defined as in \eqref{eq:fMQa}.

From \eqref{eq:estW2} we have now obtained
\begin{equation*}
    \left|T(k,\mu)\right| < \frac{H^{\omega(m, M, Q, \alpha,H)}}{c(m, M, Q, \alpha)},
\end{equation*}
where 
\begin{align*}
    c(m,M, Q, \alpha)=\;&\left(2^{m-1}(m+1)e^{1.03883m} Q^m\frac{M^{m+1}}{M-1}\right)^{-1}\\
    &\cdot \left(2\cdot 3^m e^{1.03883m} Q^m M^m \right)^{-\frac{R_1(m,M,Q,\alpha)+1}{\log f(m,M,Q,\alpha)}-1}
\end{align*}
and
\begin{align*}
     \omega(m, M, Q, \alpha, H)=\;&\left(-\frac{(m+1) \log\alpha}{\log f(m,M,Q,\alpha)}-1\right) \\
     &\cdot\left(1+\frac{\log\left(\log H+R_1(m,M,Q,\alpha)\right)}{\log H}\right)+1.
\end{align*}
The claim follows by simplifying the previous expressions.
\end{proof}

\subsubsection{Proofs for the corollaries}

Let us now move to the corollaries of Section \ref{sec:resultsMgreater}, starting with \Cref{corollary:loglogH} which illustrates the size of the exponent of $H$ more concretely.

% Corollary loglog H
\begin{proof}[Proof of Corollary \ref{corollary:loglogH}]
The claim follows from Theorem \ref{thm:mainresult2} by estimating the exponent $\omega_1(m, M, Q, \alpha, H) = \left(-\frac{(m+1) \log\alpha}{\log f(m,M,Q,\alpha)}-1\right) \left(1+\frac{\log\left(\log H+R_1(m,M,Q,\alpha)\right)}{\log H}\right)+1$.

We have
\begin{equation*}
-\frac{(m+1) \log\alpha}{\log f(m,M,Q,\alpha)}-1=\frac{\log \left(2\cdot3^m\cdot e^{1.03883m}Q^mM^m\right)}{\log f(m,M,Q,\alpha)}>0
\end{equation*}
due to assumption \eqref{eq:assumpless1}. Also
\begin{equation}
\label{eq:loglogHR}
    \frac{\log\left(\log H+R_1(m,M,Q,\alpha)\right)}{\log H} \ge 0
\end{equation}
when $H>1$ because of assumption \eqref{eq:condW1}. A lower bound for $H^{-\omega(m, M, Q, \alpha, H)}$ can therefore be found by finding an upper bound for the term in \eqref{eq:loglogHR}. For $R_1(m,M,Q,\alpha) \le 0$, an upper bound can be $\log\log H/\log H$. For $R_1(m,M,Q,\alpha) > 0$, the mean value theorem implies
$$
\frac{\log\left(\log H+R_1(m,M,Q,\alpha)\right)}{\log H}<\frac{\log\log H+R_1(m,M,Q,\alpha)/\log H}{\log H}.
$$
Using the assumptions $\log H \ge R_1(m,M,Q,\alpha)$ and $H \ge 3$, the right-hand side is equal to
$$
\frac{\log\log H}{\log H}\left(1+\frac{R_1(m,M,Q,\alpha)}{(\log H) \log \log H}\right) \le \frac{\log\log H}{\log H}\left(1+\frac{1}{\log \log 3}\right)<\frac{11.633\log\log H}{\log H}.
$$
Using this estimate for the term in \eqref{eq:loglogHR}, it follows that
$$
\omega_1(m, M, Q, \alpha, H) < \frac{(m+1)\log\alpha}{f(m,M,Q,\alpha)}+11.633 \cdot \left(1+\frac{(m+1)\log\alpha}{f(m,M,Q,\alpha)}\right)\frac{\log\log H}{\log H},
$$
and we are done.
\end{proof}

Next is the case where the exponent of $H$ is $-(m+1+\varepsilon)$:

%Corollary m+1+epsilon
\begin{proof}[Proof of \Cref{corollary:m1first}]
The goal is to prove that
$$
c_1(m,M,Q,\alpha)H^{-\omega(m,M,Q,\alpha,H)} \ge H^{-m-1-\varepsilon}.
$$
Then the result follows using \Cref{thm:mainresult2}.

First of all, since we have assumed $H \ge \left(c_1(m,M,Q,\alpha)\right)^{-3/\varepsilon}$, we also have $c_1(m,M,Q,\alpha) \ge H^{-\varepsilon/3}$. Hence we are done with the coefficient $c_1(m,M,Q,\alpha)$.

Let us now move on to the exponent $\omega_1(m,M,Q,\alpha,H)$. From assumption \eqref{eq:condHm1Big}, we get $H \ge e^{R_1(m,M,Q,\alpha)}$, and hence
\begin{align*}
    &-\omega_1(m, M, Q, \alpha, H) \\
    =\;&-\left(-\frac{(m+1) \log\alpha}{\log f(m,M,Q,\alpha)}-1\right) \left(1+\frac{\log\left(\log H+R_1(m,M,Q,\alpha)\right)}{\log H}\right)-1 \\
    \ge \;&\frac{(m+1) \log\alpha}{\log f(m,M,Q,\alpha)}+\left(\frac{(m+1) \log\alpha}{\log f(m,M,Q,\alpha)}+1\right)\frac{\log\log (2H)}{\log H}.
\end{align*}    
By assumption \eqref{eq:m1epsilonpart}, the last line is
\begin{equation}
\label{eq:omegam1}
    \ge -(m+1)-\frac{\varepsilon}{3}+\left(-m-\frac{\varepsilon}{3}\right)\frac{\log\log (2H)}{\log H}.
\end{equation}
The assumption
$$
H \ge \frac{1}{2}e^{-\left(3m/\varepsilon+1\right)W_{-1}\left(-2^{-\varepsilon/(3m+\varepsilon)}\frac{\varepsilon}{3m+\varepsilon}\right)}
$$
now implies
$$
\frac{\log\log(2H)}{\log H} \le \frac{\varepsilon}{3m+\varepsilon}.
$$
Hence the last term in estimate \eqref{eq:omegam1} is at least $-\frac{\varepsilon}{3}$, and the claim follows.
\end{proof}

Finally, we prove \Cref{corollary:main2Simp}:
%Corollary M>1, easier version
\begin{proof}[Proof of \Cref{corollary:main2Simp}]
The claim can be proved similarly as \Cref{thm:mainresult2}.
Let $t$ be as in the proof of \Cref{thm:mainresult2}. Since condition \eqref{eq:condlogf} holds and we have assumed
\begin{equation*}
\log H \ge \max\left\{1.03883m, (m+1)\log M, \log(m+2), m\log Q\right\},
\end{equation*}
we also have
\begin{equation*}
\log H \ge \max\left\{\log(m+1), (m-1)\log 2\right\}.
\end{equation*}
Thus $t \le 7\log H -1$. Due to Lemma \ref{lemma:estLambertW}, condition \eqref{eq:condlogf}, and since $x-1>\log x$ for all $x>0$, we can estimate 
\begin{align*}
&\left\lceil \frac{1}{\log \left( 2\cdot 3^me^{1.03883m}  Q^mM^m \alpha^{m+1} \right)} \vphantom{W_{-1}\left(\frac{\log \left( 2\cdot 3^me^{1.03883m}  Q^mM^m \alpha^{m+1} \right)}{2^{m-1}e^{1.03883m}Q^m H(m+1)\cdot \frac{M^{m+1}}{M-1} \cdot (m+2) \alpha}\right)}\right. \\
&\left. \cdot W_{-1}\left(\frac{\log \left( 2\cdot 3^me^{1.03883m}  Q^mM^m \alpha^{m+1} \right)}{2^{m-1}e^{1.03883m}Q^m H(m+1)\cdot \frac{M^{m+1}}{M-1} \cdot (m+2) \alpha}\right) \right\rceil \\
< \;&\frac{-\log(t+1)-t-2+\log(e-1)}{\log \left( 2\cdot 3^me^{1.03883m}  Q^mM^m \alpha^{m+1} \right)}+1 \\
< \;&2t+3 \\
\le \;&14\log H+1.
\end{align*}

The claim follows from \eqref{eq:estW2} by using the bound $k < 14\log H+1$.
\end{proof}

% Proof M<1
\subsection{Proofs of the case $M<1$}

Let us now move to the case $M < 1$. We start with the first result in \Cref{sec:resultsSmall}:

\begin{proof}[Proof of Theorem \ref{thm:mainresult3}]
We proceed similarly as in the proof of Theorem \ref{thm:mainresult2}, taking into account the different estimates for $\left| B_{k,\mu,j}(1)\right|$. Now, by estimate \eqref{eq:maxBMle1} and using $k \ge 1$, we have
\begin{equation}\label{eq:estTkmu}
\begin{split}
|T(k,\mu)| &\le Q^{mk+m}\left( \sum_{j=0}^m |\lambda_j| \right) \max_{0 \le j \le m} \left| B_{k,\mu,j}(1)\right| \\
&\le Q^{mk+m} (m+1)H 2^{k+m-1} 3^{mk} d_{mk+m}(mk+m) \\
&\le Q^m (m+1)H 2^{m-1} e^{1.03883m} k\left(m+\frac{m}{k}\right) \left( 2 \cdot 3^m Q^m e^{1.03883m} \right)^k \\
&\le 2m(m+1)Q^m  H 2^{m-1} e^{1.03883m} k e^{k \log \left( 2 \cdot 3^m Q^m e^{1.03883m} \right)}.
\end{split}
\end{equation}
We concentrate on the terms containing the number $k$ on the last row of the previous inequality.

%Just as in the proof of Theorem \ref{thm:mainresult2}, the consideration splits into two cases:

%Assume first that condition \eqref{eq:condH5} does not hold. By Lemma \ref{lemma:contr3}, we may choose $k=1$, giving
%$$
%|T(k,\mu)| \le H Q^{2m} 2^{m+1} 3^m m(m+1)e^{2.07766m}.
%$$

%Suppose then that conditions \eqref{eq:condH5} and \eqref{eq:condH3} hold. 
Along Lemma \ref{lemma:contr3}, we choose
\begin{align*}
k = \,&\left\lceil -2W_{-1}\left( -\frac{1}{2} \sqrt{\frac{\left( \log \left( 2\cdot 3^me^{1.03883m}Q^mM^m \alpha^{m+1} \right) \right)^2}{2^m e^{1.03883m}Q^m H(m+1) \cdot m(m+2) \alpha}} \right) \right.\\
&\left. \cdot \left(\log \left( 2^{-1}\cdot 3^{-m}e^{-1.03883m}Q^{-m}M^{-m} \alpha^{-(m+1)} \right) \right)^{-1} \vphantom{ -2W_{-1}\left( -\frac{1}{2} \sqrt{\frac{\left( \log \left( 2\cdot 3^me^{1.03883m}Q^mM^m \alpha^{m+1} \right) \right)^2}{2^m e^{1.03883m}Q^m H(m+1) \cdot m(m+2) \alpha}} \right)}\right\rceil.
\end{align*}
Since the formula inside the ceiling function is positive, we have $k \ge 1$.
Now
\begin{align*}
k \le \;&-2W_{-1}\left( -\frac{1}{2} \sqrt{\frac{\left( \log \left( 2\cdot 3^me^{1.03883m}Q^mM^m \alpha^{m+1} \right) \right)^2}{2^m e^{1.03883m}Q^m H(m+1) \cdot m(m+2) \alpha}} \right) \\
&\cdot \left(\log \left( 2^{-1}\cdot 3^{-m}e^{-1.03883m}Q^{-m}M^{-m} \alpha^{-(m+1)} \right) \right)^{-1} +1.
\end{align*}

First we estimate $W_{-1}$ using Lemma \ref{lemma:estLambertW}:
\begin{align*}
&W_{-1}\left( -\frac{1}{2} \sqrt{\frac{\left( \log \left( 2\cdot 3^me^{1.03883m}Q^mM^m \alpha^{m+1} \right) \right)^2}{2^m e^{1.03883m}Q^m H(m+1) \cdot m(m+2) \alpha}} \right) \\
= \;&W_{-1} \left(-e^{-t-1} \right) \\
> \;&-\log(t+1) - t -2 +\log (e-1),
\end{align*}
where
\begin{align*}
t = \;&\log \left( \frac{2 \sqrt{2^m e^{1.03883m}Q^m H(m+1) \cdot m(m+2) \alpha}}{\log \left( 2^{-1}\cdot 3^{-m} e^{-1.03883m}Q^{-m} M^{-m} \alpha^{-(m+1)} \right)} \right) -1 \\
= \;&\log \left( 2^{\frac{m}{2}+1} e^{0.519415m}Q^{\frac{m}{2}} H^{\frac{1}{2}}m^{\frac{1}{2}}(m+1)^{\frac{1}{2}}(m+2)^{\frac{1}{2}} \alpha^{\frac{1}{2}} \right) \\
&- \log \log \left( 2^{-1}\cdot 3^{-m} e^{-1.03883m}Q^{-m} M^{-m} \alpha^{-(m+1)} \right) -1 \\
= \;&\frac{1}{2} \log H + \left( \frac{m}{2}+1 \right) \log 2 + 0.519415m + \frac{m}{2} \log Q \\
&+ \frac{1}{2} \left( \log m + \log (m+1) + \log (m+2) + \log \alpha \right) \\
&- \log \log \left( 2^{-1}\cdot 3^{-m} e^{-1.03883m}Q^{-m} M^{-m} \alpha^{-(m+1)} \right) -1.
\end{align*}

According to condition \eqref{eq:condH3}, we have
$$
\log H > \max \left\{ \left( \frac{m}{2}+1 \right) \log 2, 0.519415m, \frac{m}{2} \log Q, \log (m+2) \right\}.
$$
It is also assumed that $f(m, M, Q, \alpha) \ge e$, so
$$
- \log \log \left( 2^{-1}\cdot 3^{-m} e^{-1.03883m}Q^{-m} M^{-m} \alpha^{-(m+1)} \right) \le 0.
$$
Hence
$
t < 5\log H -1.
$

Combining the results above, we get
\begin{equation}\label{eq:estk2}
\begin{split}
k &< \frac{2t + 2\log (t+1) +4 -2\log (e-1)}{\log \left( 2^{-1}\cdot 3^{-m}e^{-1.03883m}Q^{-m}M^{-m} \alpha^{-(m+1)} \right)} +1 \\
&< \frac{2(5\log H -1) + 2\log (5\log H) +4 -2\log (e-1) + \log f(m, M, Q, \alpha)}{\log f(m, M, Q, \alpha)} \\
&< \frac{10 \log H + 2\log \log H + 2\log 5  +2 -2\log (e-1) + \log f(m, M, Q, \alpha)}{\log f(m, M, Q, \alpha)} \\
&< \frac{10 \log H + 2\log \log H +5 + \log f(m, M, Q, \alpha)}{\log f(m, M, Q, \alpha)} \\
&= \frac{\left( 10 + \frac{2\log \log H}{\log H} \right) \log H + 5+ \log f(m, M, Q, \alpha)}{\log f(m, M, Q, \alpha)}.
\end{split}
\end{equation}
In \eqref{eq:condH3}, we assumed that $H>e^e$, so $\frac{2\log \log H}{\log H} < \frac{2}{e}$. It follows that
\begin{equation}\label{eq:boundforKsmall}
k<\frac{11 \log H+5 + \log f(m, M, Q, \alpha)}{\log f(m, M, Q, \alpha)}.
\end{equation}

Now we bound $e^{k \log \left( 2 \cdot 3^m Q^m e^{1.03883m} \right)}$ using \eqref{eq:boundforKsmall}:
\begin{align*}
e^{k \log \left( 2 \cdot 3^m Q^m e^{1.03883m} \right)} &<\left( e^{11\log H +5+ \log f(m, M, Q, \alpha)} \right)^{\frac{\log \left( 2 \cdot 3^m Q^m e^{1.03883m} \right)}{\log f(m, M, Q, \alpha)}} \\
&=\left(H^{11}\cdot e^5 \cdot f(m, M, Q, \alpha) \right)^{\frac{\log \left( 2 \cdot 3^m Q^m e^{1.03883m} \right)}{\log f(m, M, Q, \alpha)}}.
\end{align*}

It remains to estimate the rest of the terms on the last line of estimate \eqref{eq:estTkmu}. For that, we further estimate $k$ using \eqref{eq:estk2} and the fact that $\log f(m, M, Q, \alpha) \ge 1$:
\begin{align*}
k &< \frac{\left( 10 + \frac{2\log \log H}{\log H} + \frac{5}{\log H} \right) \log H + \log f(m, M, Q, \alpha)}{\log f(m, M, Q, \alpha)} \\
&< \left( 10 + \frac{2\log \log H}{\log H} + \frac{5}{\log H} \right) \log H + 1 \\
&< \left( 10 + \frac{2\log \log H}{\log H} + \frac{6}{\log H} \right) \log H \\
&< 13\log H,
\end{align*}
where we again made use of the fact $\frac{2\log \log H}{\log H} < \frac{2}{e}$.

Let us finally continue from \eqref{eq:estTkmu}. According to the previous estimates, we have
\begin{align*}
|T(k,\mu)| \le \;&Q^m m(m+1) H 2^m e^{1.03883m} k e^{k \log \left( 2 \cdot 3^m Q^m e^{1.03883m} \right)} \\
< \;&Q^m m(m+1) H 2^m e^{1.03883m} \cdot 13\log H  \\
&\cdot \left(H^{11}e^5 \cdot f(m, M, Q, \alpha) \right)^{\frac{\log \left( 2 \cdot 3^m Q^m e^{1.03883m} \right)}{\log f(m, M, Q, \alpha)}} \\
= \;& 13m(m+1) (2Qe^{1.03883})^m  \cdot\left(e^5 \cdot f(m, M, Q, \alpha)\right)^{\frac{\log \left( 2 \cdot 3^m Q^m e^{1.03883m} \right)}{\log f(m, M, Q, \alpha)}} \\
&\cdot H^{\frac{11\log \left( 2 \cdot 3^m Q^m e^{1.03883m} \right)}{\log f(m, M, Q, \alpha)}+\frac{\log\log H}{\log H}+1} \\
= \;&\left(26m(m+1)\left(6Q^2e^{2.07766}\right)^me^{\frac{5\log \left( 2 \cdot 3^m Q^m e^{1.03883m} \right)}{\log f(m, M, Q, \alpha)}}\right) \\
&\cdot H^{\frac{11\log \left( 2 \cdot 3^m Q^m e^{1.03883m} \right)}{\log f(m, M, Q, \alpha)}+\frac{\log\log H}{\log H}+1} \\
= \;& \frac{H^{\omega(m,M,Q,\alpha,H)}}{c(m, M, Q, \alpha)},
\end{align*}
where we denote
$$
c(m, M, Q, \alpha) := \left(26m(m+1)\left(6Q^2e^{2.07766}\right)^me^{\frac{5\log \left( 2 \cdot 3^m Q^m e^{1.03883m} \right)}{\log f(m, M, Q, \alpha)}}\right)^{-1}
$$
and
$$
\omega(m, M, Q, \alpha, H) := 
\frac{11\log \left( 2 \cdot 3^m Q^m e^{1.03883m} \right)}{\log f(m, M, Q, \alpha)}+\frac{\log\log H}{\log H}+1.
$$
The wanted result follows from the upper bound for $T(k,\mu)$.

\end{proof}

Finally we have \Cref{thm:mainresult4}, where the exponent of $H$ is of the best possible form:
%m+1+epsilon for M<1
\begin{proof}[Proof of \Cref{thm:mainresult4}]
The idea of the proof is to divide the estimate for the term $|T(k,\mu)|$ into different parts based on how they depend on the number $H$. The proof is similar to the proof of \Cref{corollary:m1first}.

According to the first three paragraphs of the proof of \Cref{thm:mainresult3}, we can find an upper bound \eqref{eq:estTkmu} for the term $|T(k,\mu)|$. We bound the number $k$ which depends on the number $t$, where both $k$ and $t$ are given as in the first three paragraphs of the proof of \Cref{thm:mainresult3}. Since $\log f(m,M,Q,\alpha) \ge 1$, we have
\begin{align*}
t = \;&\frac{1}{2} \log H + \left( \frac{m}{2}+1 \right) \log 2 + 0.519415m + \frac{m}{2} \log Q \\
&+ \frac{1}{2} \left( \log m + \log (m+1) + \log (m+2) + \log \alpha \right) \\
&- \log \log \left( 2^{-1}\cdot 3^{-m} e^{-1.03883m}Q^{-m} M^{-m} \alpha^{-(m+1)} \right) -1 \\
< \;& \frac{1}{2} \log H + \left( \frac{m}{2}+1 \right) \log 2 + 0.519415m + \frac{m}{2} \log Q \\
&+ \frac{1}{2} \left( \log m + \log (m+1) + \log (m+2) + \log \alpha \right)-1 \\
= \;& \frac{1}{2}\log H+R_2(m,Q,\alpha)-1.
\end{align*}
Note that $R_2(m,Q,\alpha) \ge 0$ since $Q\alpha \ge 1$ because of the assumption $M<1$.

Combining the previous estimate and the results in the first three paragraphs of the proof of \Cref{thm:mainresult3}, and noting that from \eqref{eq:assumpHepsilon} it also follows $H \ge e^{2R_2(m,Q,\alpha)}$, we get
\begin{equation}\label{eq:boundfork}
\begin{split}
k &< \frac{2t + 2\log (t+1) +4 -2\log (e-1)}{\log \left( 2^{-1}\cdot 3^{-m}e^{-1.03883m}Q^{-m}M^{-m} \alpha^{-(m+1)} \right)} +1 \\
&< \frac{\log H+2R_2(m,Q,\alpha)+2\log\log H+1}{\log f(m,M,Q,\alpha)}+1.
\end{split}
\end{equation}

%Let us now set the value of $k$ to be in \eqref{eq:estTkmu} as in the last line on the previous estimate. 
Keeping mind that $\log f(m,M,Q,\alpha) \ge 1$ and $H \ge \max\{e^{2R_2(m,Q,\alpha)},e^e\}$ and using the bound \eqref{eq:boundfork} for $k$, we get from \eqref{eq:estTkmu} that
\begin{align*}
|T(k,\mu)| \le \;&Q^m m(m+1) H 2^m e^{1.03883m} k e^{k \log \left( 2 \cdot 3^m Q^m e^{1.03883m} \right)} \\
< \;& (2Q)^m m(m+1)e^{1.03883m}H \cdot 2\left(\log H+\log\log H+1\right) \\
&\cdot e^{\left(\frac{\log H+2R_2(m,Q,\alpha)+2\log\log H+1}{\log f(m,M,Q,\alpha)}+1\right)\log \left( 2 \cdot 3^m Q^m e^{1.03883m} \right)} \\
\le \;&2^{m+4+2R_2(m,Q,\alpha)}\cdot9^{R_2(m,Q,\alpha)+1}\left(e^{1.03883}Q\right)^{(2R_2(m,Q,\alpha)+3)m}m(m+1)\\
&\cdot H^{\frac{\log \left( 2 \cdot 3^m Q^m e^{1.03883m} \right)}{\log f(m,M,Q,\alpha)}\left(1+\frac{2\log\log H}{\log H}\right)+\frac{\log\log H+\log\log\log H}{\log H}+1}.
\end{align*}

By assumption \eqref{eq:assumpHepsilon}, the second last line on the previous estimate is at most $H^{\varepsilon/3}$. Further, according to
assumption \eqref{eq:flogmeps}, we have
$$
\frac{\log \left( 2 \cdot 3^m Q^m e^{1.03883m} \right)}{\log f(m,M,Q,\alpha)}+1=-\frac{\log(M^m\alpha^{m+1})}{\log f(m,M,Q,\alpha)} \le m+1+\frac{\varepsilon}{3}.
$$
and by assumption \eqref{eq:assumpHepsilonSecond}, we have $H \ge e^e$, so
\begin{align*}
&\frac{\log \left( 2 \cdot 3^m Q^m e^{1.03883m} \right)}{\log f(m,M,Q,\alpha)}\cdot\frac{2\log\log H}{\log H}+\frac{\log\log H+\log\log\log H}{\log H} \\
<\;&\left(2\left(m+\frac{\varepsilon}{3}\right)+2\right)\frac{\log\log H}{\log H} \\
<\;&\frac{\varepsilon}{3}.
\end{align*}
Putting these together, we get the wanted result.
\end{proof}

%EXAMPLES
\section{Examples}
\label{sec:examples}

Let us finally take a look at a few examples of the results, starting with a case where $m=1$ and we use \Cref{corollary:m1first}:
%Example m=1
\begin{example}
\label{example:pm1}
Let $\varepsilon=0.1$. Let $p$ be a prime, $a\ge 1$, and $\lambda_0$ and $\lambda_1$ be integers such that $p^a \ge (6e^{1.03883})^{21} \approx 7\cdot10^{25}$ and 
$$
H=\max\{|\lambda_0|, |\lambda_1|\} \ge \max\left\{\left(c_1(1,p^a,1,1/p^a)\right)^{-31},\frac{1}{2}e^{-31W_{-1}\left(-2^{-1/31}/31\right)} \right\}.
$$
For example, if $p=11$ and $a=25$, it is sufficient that $H \ge 3\cdot 10^{1672}$.

By \Cref{corollary:m1first}, we have
\begin{equation*}
    \left|\lambda_1\boldsymbol{\log}\left(1+p^a\right)-\lambda_0 \right|_p> H^{-2-0.1}.
\end{equation*}

\end{example}

In the next example, we consider the case $m=2$ and use \Cref{thm:mainresult2}:

\begin{example}
\label{example:pfirst}
Let $p \ge 149$ be a prime number and $\lambda_0,\lambda_1$, and $\lambda_2$ be integers. By \Cref{thm:mainresult2}, we have
\begin{equation}
\label{eq:ptwoFirst}
\begin{split}
    &\left|\lambda_0+\lambda_1\boldsymbol{\log}(1+p)+\lambda_2\boldsymbol{\log}(1-p)\right|_p \\
    >\;&\frac{2}{9e^{1.07766} p^3\log f(2,p,1,1/p)} \cdot e^{\frac{-3\log p}{\log f(2,p,1,1/p)}\left(R_1(2,p,1,1/p)+1\right)} \\
    &\cdot H^{-\left(\frac{3 \log p}{\log f(2,p,1,1/p)}-1\right)\left(1+\frac{\log\left(\log H+R_1(2,p,1,1/p)\right)}{\log H}\right)-1},
\end{split}     
\end{equation}
where $H=\max\{|\lambda_0|,|\lambda_1|,|\lambda_2|\}$,
\begin{equation*}
    f(2,p,1,1/p)=\frac{p}{18e^{2.07766}},
\end{equation*}
and
\begin{equation*}
    R_1(2,p,1,1/p) = 2.07766+ \log24+ \log \frac{p^2}{p-1} -\log \log f(2, p, 1, 1/p).
\end{equation*}

For example, if $p=149$, then the right-hand side of \eqref{eq:ptwoFirst} is
\begin{align*}
    &\frac{2e^{\frac{-3\log 149}{\log \left(149/(18e^{2.07766})\right)}\left(3.07766+ \log \frac{6\cdot149^2}{37} -\log \log \frac{149}{18e^{2.07766}}\right)}}{9\cdot149^3e^{1.07766}\log \left(149/(18e^{2.07766})\right)} \\
    &\cdot H^{-\left(\frac{3 \log 149}{\log \left(149/(18e^{2.07766})\right)}-1\right)\left(1+\frac{\log\left(\log H+2.07766+ \log \frac{6\cdot149^2}{37} -\log \log \frac{149}{18e^{2.07766}}\right)}{\log H}\right)-1} \\
    \approx \;&\frac{1}{10^{2050}}H^{-418\cdot\left(1+\frac{\log\left(\log H+14\right)}{\log H}\right)}.
\end{align*}

\end{example}

Let us finally look at an example concerning the case $M<1$:

\begin{example}
We consider the linear form
$$
\lambda_0 + \lambda_1 \boldsymbol{\log} \left( 1+\frac{p^a}{1+p^a} \right), \quad \lambda_0, \lambda_1 \in \mathbb{Z},
$$
where the numbers $a$ and $H=\max \{\lambda_0, \lambda_1\}$ are assumed to be such that
$$
p^a > 6e^{2.03883} \approx 46.09,
$$
and
$$
 H \ge \left\{\frac{e^2 \left( \log \left(  \frac{6 e^{1.03883}}{p^{a}} \right) \right)^2}{48e^{1.03883} \left(1+\frac{1}{p^a}\right)},\sqrt{1+p^a},e^e \right\}.
$$
%and
%$$
%\log H > \max \left\{ \frac{1}{2} \log (1+p^a), e \right\}.
%$$

The conditions of Theorem \ref{thm:mainresult3} are fulfilled; in addition we need
$$
\frac{\log \left( \frac{1}{p^a(1+p^a)} \right)}{\log \left(  \frac{6 e^{1.03883}}{p^{a}} \right)} \ge -2-\frac{\varepsilon}{3},
$$
\begin{align*}
    H \ge \left(2^{6+2R_2(m,Q,\alpha)} \cdot 9^{R_2(m,Q,\alpha)+1}\left(e^{1.03883}(1+p^a)\right)^{2R_2(m,Q,\alpha)+3} \right)^{3/\varepsilon}
\end{align*}
with
$$
R_2 = 2\log 2 + \frac{1}{2} \log 3 + 0.519415 + \frac{1}{2} \left( \log (1+p^a) - \log (p^a) \right),
$$
and
\begin{align*}
    H \ge e^{-\frac{6(2+\varepsilon/3)}{\varepsilon}W_{-1}\left(-\frac{\varepsilon}{6(2+\varepsilon/3)}\right)}.
\end{align*}
Then Theorem \ref{thm:mainresult4} gives
$$
\left| \lambda_0 + \lambda_1 \boldsymbol{\log} \left( 1+\frac{p^a}{1+p^a} \right) \right|_p > \frac{1}{H^{2+\varepsilon}}.
$$

For example, for $\varepsilon =0.1$, $p=11$, and $a=25$ as in Example \ref{example:pm1}, having $H \ge 3.6 \cdot 10^{6482}$ gives the lower bound $\frac{1}{H^{2+0.1}}$.
\end{example}

\section*{Acknowledgements}
We would like to thank Tapani Matala-aho for his support during the research process. We are also grateful to the anonymous referee for their comments that helped to improve the manuscript.

\newpage

\appendix

\section{Appendix: Comparison to the results of others}

\subsection{V\"a\"an\"anen \& Xu}\label{app:VaananenXu}

V\"a\"an\"anen and Xu \cite{Vaananen1988} use Padé approximations of the second kind to prove a lower bound for a linear form in the values of $G$-functions at points from an algebraic number field. In their result, the exponent of $H$ is the best possible form $m+1+\varepsilon$ whereas in our results it is of that type only in some of the cases. However, in order to apply the result proved by V\"a\"an\"anen and Xu, one needs to know the coefficients of the system of linear differential equations that the functions $y_i(z)$ satisfy (see the theorem below); in our results these are not needed. 

In the $p$-adic case they prove:

\begin{theorem} \cite[Corollary 2]{Vaananen1988}
Let $\lambda_1 y_1(z) + \ldots + \lambda_m y_m(z)$ be a linear form in $\mathbb{Q}G$-functions $y_i(z)$, where the coefficients $\lambda_i$ are integers. Let $\frac{a}{b}$ be a rational number satisfying $\gcd (a,b)=1$, $\frac{a}{b}T\left(\frac{a}{b}\right) \neq 0$, and
$$
\left|\frac{a}{b}\right|_p \le \frac{1}{2C}.
$$
Here $T(z)$ is the common denominator of the coefficients of the system of linear differential equations that the functions $y_i$ satisfy, and $C \ge 1$ is a constant related to the size of the coefficients of the functions $y_i$ present in the definition of $G$-functions.

Let $u>0$ and $\varepsilon <1$ be given. There exists an explicit constant $C_1$ depending only on $u$, $\varepsilon$, and the functions $y_i$ such that if
$$
\left( \max \left\{ |a|,|b| \right\} \right)^{u\varepsilon} > C_1 \left( \left|\frac{a}{b}\right|_p \max \left\{ |a|,|b| \right\} \right)^{(m+1)(m+\varepsilon)},
$$
then
$$
\left| \lambda_0 + \lambda_1 y_1 \left( \frac{a}{b} \right) + \ldots + \lambda_m y_m \left( \frac{a}{b} \right) \right|_p > \frac{1}{H^{m+1+\varepsilon}}
$$
for all
$$
H=\max \left\{ \lambda_0, \lambda_1, \ldots, \lambda_m \right\} > \max \left\{ C_1, \frac{2B}{(1-u)\varepsilon} \right\}.
$$

Above
\begin{multline*}
B= \frac{(\mu_1 +1 +2\delta ms) \log \left( \prod_p \max \left\{ 1, \left|\frac{a}{b}\right|_p \right\} \right)}{\mu_1 \log \left( \prod_p \max \left\{ 1, \left|\frac{a}{b}\right|_p \right\} \right) - A} \cdot \log \left( 4^{\frac{1}{\delta m}} (m+1) \right) \\
+ A + (1+2\delta ms) \log \left( \prod_p \max \left\{ 1, \left|\frac{a}{b}\right|_p \right\} \right),
\end{multline*}
where $\delta$ and $\delta_1$ are given numbers such that $0 < \delta < \frac{1}{3m^2(s+1)}$ and $0<\delta_1 <1$, $s$ is the maximum of the degrees of $T(z)$ and of $T(z)$ multiplied by the coefficients of the mentioned differential system, and
$$
\mu_1 = \frac{1}{m} + 3\delta \delta_1 m -\delta m (2s+3) -\delta_1 \left( 1+\frac{1}{m} \right),
$$
$$
A = \frac{3\left( 1+\frac{1}{m} \right)\log C +2}{\delta m}+2\delta m\log \left( \prod_p \max \left\{ 1, |T|_p \right\} \right).
$$
\end{theorem}

%Yu
\subsection{Yu}\label{app:Yu}
In his papers \cite{Yu1989,Yu1990,Yu1994,Yu1998,Yu1999,Yu2007}, Yu has established many results for linear forms in $p$-adic logarithms. His results are mainly more general than ours since they handle algebraic numbers $\alpha_1-1, \alpha_2-1, \ldots, \alpha_n-1$ instead of rational numbers, even though he also proves results specially for rational numbers.

Using formulation \eqref{eq:ourForm} in Yu's results, the exponent $\omega$ is at least of size $e^{2m}m^{3/2}$ when $H$ (or $B$ in Yu's notation) is large enough whereas our results are of the best possible type $m+1+\varepsilon$ in some of the cases. 

Precisely, in \cite{Yu2007} Yu proves the following: %T\"am\"a muotoilu l\"oytyi Yann Bugeaudin artikkelista "On the digital representation of integers with bounded prime factors" (Osaka J. Math., 2018).
\begin{theorem} \cite{Yu2007} 
Let $p$ be a prime number and $m \ge 2$ an integer. Let $\frac{x_1}{y_1}, \ldots, \frac{x_m}{y_m}$ be non-zero rational numbers and $A_1, \ldots, A_m$ real numbers such that
$$
A_i \ge \max \{|x_i|, |y_i|, e\}, \quad 1 \le i \le m.
$$
Let $b_1, \ldots, b_m$ be non-zero integers such that $\left(\frac{x_1}{y_1}\right)^{b_1} \cdots \left(\frac{x_m}{y_m}\right)^{b_m} \neq 1$. Let $B$ and $B_m$ be real numbers such that
$$
B \ge \max \{|b_1|, \ldots, |b_m|, 3\}, \quad B \ge B_m \ge |b_m|.
$$
Suppose that $v_p(b_n) \le v_p(b_i)$ for all $1 \le i \le m$. Let $\delta$ be a real number with $0 < \delta \le \frac{1}{2}$. Then
\begin{multline}\label{eq:Yu2007}
v_p \left( \left(\frac{x_1}{y_1}\right)^{b_1} \cdots \left(\frac{x_m}{y_m}\right)^{b_m} - 1 \right) < (16e)^{2(m+1)} m^\frac{3}{2} (\log (2m))^2 \cdot \frac{p}{(\log p)^2} \\
\cdot \max \left\{ (\log A_1) \cdots (\log A_m) (\log T), \frac{\delta B}{B_m} \right\},
\end{multline}
where
$$
T = \frac{2B_n}{\delta} \cdot e^{(m+1)(6m+5)} p^{m+1} (\log A_1) \cdots (\log A_{m-1}).
$$
\end{theorem}

As Yu explains in \cite[Section 1.1]{Yu1989}, we have
$$
v_p \left( \left(\frac{x_1}{y_1}\right)^{b_1} \cdots \left(\frac{x_m}{y_m}\right)^{b_m} - 1 \right) = v_p \left( b_1 \log \left( \frac{x_1}{y_1} \right) + \ldots + b_m \log \left( \frac{x_m}{y_m} \right) \right)
$$
when $v_p \left( \frac{x_i}{y_i}-1 \right) > \frac{1}{p-1}$ for all $1 \le i \le m$. Hence the upper bound in \eqref{eq:Yu2007} translates into a lower bound
\begin{multline*}
\left| b_1 \log \left( \frac{x_1}{y_1} \right) + \ldots + b_m \log \left( \frac{x_m}{y_m} \right) \right|_p \\
> p^{- (16e)^{2(m+1)} m^\frac{3}{2} (\log (2m))^2 \cdot \frac{p}{(\log p)^2} \cdot \max \left\{ (\log A_1) \cdots (\log A_m) (\log T), \frac{\delta B}{B_m} \right\}}.
\end{multline*}

\end{document}